\documentclass[a4paper]{article}

\usepackage[english]{babel}
\usepackage[utf8x]{inputenc}
\usepackage[T1]{fontenc}
\usepackage{upgreek}
\usepackage{gensymb}
\makeatletter
\newcommand{\ssymbol}[1]{^{\@fnsymbol{#1}}}
\makeatother

\usepackage{amssymb,amsmath,amsfonts,amsthm,latexsym,enumerate,url,cases}
\usepackage{graphicx}
\usepackage[colorinlistoftodos]{todonotes}
\usepackage[colorlinks=true, allcolors=blue]{hyperref}

\usepackage[a4paper,top=3cm,bottom=2cm,left=3cm,right=3cm,marginparwidth=1.75cm]{geometry}

\usepackage{authblk}
\title{The Large Deviation of Semilinear Stochastic Partial Differential Equation Driven by Brownian Sheet\thanks{The work is supported in part by the NSFC Grant No. 12171084 and the fundamental Research Funds for the Central Universities No. 2242022R10013.}}
\author{Qiyong Cao}
\author[$\ssymbol{2}$]{Hongjun Gao}
\affil[$\ssymbol{2}$]{  School of Mathematical Sciences, Nanjing Normal University, 210023, Nanjing, China}
\affil[$\ssymbol{2}$]{  School of mathematics, Southeast University, 211189,  Nanjing, China}
\usepackage{fullpage}
\usepackage{amssymb,amsmath,amsfonts,amsthm,eufrak,latexsym,enumerate,url,cases}
\usepackage{mathrsfs}
\numberwithin{equation}{section}
\usepackage{hyperref}
\usepackage{indentfirst}
\usepackage{hyperref}
\hypersetup{colorlinks=true,citecolor=blue,linkcolor=blue,urlcolor=blue}

\newtheorem{theorem}{Theorem}[section] %
\newtheorem{lemma}{Lemma}[section] %
\newtheorem{definiton}{Definition}[section] %
\newtheorem{remark}{Remark}[section] %

\usepackage{cite}
\usepackage{hyperref} 
\hypersetup{hypertex=true,
            colorlinks=true,
            linkcolor=blue,
            anchorcolor=blue,
            citecolor=blue}
\usepackage{bm}
\usepackage{authblk}
\usepackage{indentfirst} 
\begin{document}
\maketitle

\begin{abstract}
We prove  the large deviation principle(LDP) for the law of the one-dimensional semi-linear stochastic partial differential equations  driven by a nonlinear multiplicative noise. Firstly, combining the energy estimate and  approximation procedure, we obtain the existence of the global solution. Secondly, the large deviation principle is obtained via the weak convergence method.
\end{abstract}


\section{Introduction}
In this paper, we consider the one-dimensional semi-linear stochastic partial differential equations,
\begin{equation}\label{1.1}
\left\{\begin{aligned}
\frac{\partial u^{\epsilon}}{\partial t}(t,x)&=\frac{\partial^{2} u^{\epsilon}}{\partial x^{2}}(t,x)+\frac{\partial }{\partial x}g(t,x, u^{\epsilon}(t,x))+f(t,x,u^{\epsilon}(t,x))+\sqrt{\epsilon}\sigma(t,x,u^{\epsilon}(t,x))\frac{\partial^{2} W}{\partial t \partial x}(t,x), \\
u^{\epsilon}(t,0)&=u^{\epsilon}(t,1)=0, \\
u^{\epsilon}(0,x)&=\eta(x),
\end{aligned}\right.
\end{equation}
where $\frac{\partial W(t,x)}{\partial x}$ denotes the Brownian sheet which is defined on the probability space $(\Omega,\mathcal{F},\mathbb{P})$, the corresponding filtration to Brownian sheet is $\mathcal{F}_{t}$. In fact, $\frac{\partial W(t,x)}{\partial x}$  is a two parameter center Gaussian process which has a covariance $\mathbb{E}\left\{\frac{\partial W(t_1,x_1)}{\partial x}\frac{\partial W(t_2,x_2)}{\partial x}\right\}=\min\{t_1,t_2\}\min\{x_{1},x_{2}\},t_{1},t_2\in[0,T],x_{1},x_{2}\in[0,1]$. Furthermore, it can be written an explicit form
\begin{equation*}
\frac{\partial W(t,x)}{\partial x}=\sum_{k\geq 1}h_k(x)w_{k}(t),
\end{equation*}
where $\{w_{k}(t),k\geq 1\}_k$ is a collection of independent standard Brownian motion, and $h_k(x)=\int_{0}^{x}\mathfrak{h}_{k}(r)dr$, $\left\{\mathfrak{h}_{k}(x), k \geq 1, x \in [0,1]\right\}$ is an orthonormal basis in the space $L^2{[0,1]}$, it is easy to see that $\sum_{k}h_{k}^2(x)=x$, more properties on Brownian sheet can be found in \cite[Chapter 1]{MR3674586}.  The functions $f=f(t,x,r),g=g(t,x,r),\sigma=\sigma(t,x,r)$ are the Borel functions of
$(t,x,r)\in [0,T]\times [0,1] \times \mathbb{R}$. We further assume that $\sigma$ is Lipschitz, $g$ and $f$ quadratic and linear growth respectively.

 The equation \eqref{1.1} contains two main types of equations, one is the stochastic reaction diffusion equation $(g\equiv 0)$ and the other is the stochastic Burger's equation $(f\equiv 0, g(u)=u^{2})$ which  describes the turbulence phenomenon in the fluid. For the study of semi-linear stochastic evolution equations driven by the space-time white noise, Gy\"{o}ngy\cite{MR1608641} first used the approximation methods to study the case of bounded diffusion term coefficients  and obtained the existence of uniqueness of solutions of the equations as well as the comparison principle.  After that, Cardon-Weber\cite{MR1720097} established the large deviation principle for such a class of semi-linear parabolic equations based on the time discretization method. In addition, Foondun and Selayeshar\cite{MR3575422} generalized Selayeshar's  results \cite{MR4246020} from Burgers' equation to the general (i.e., semi-linear evolution equations) using a weak convergence method. Finally, for this class of equations, besides the large deviation principle, Xiong et al.\cite{MR4258396} and Hu et al.\cite{MR4143604} studied the central limit theorem for $ C([0,T];L^{2}([0,1]))$-valued solutions of the equation as well as the behavior of moderate deviations. It is worth noting that the diffusion terms are bounded in the above results, while the diffusion coefficients of the equations we study in this paper are locally bounded, and we try to establish the large deviation behavior to the solutions of the equations driven by a
Brownian sheet.

In this paper, we first establish  the well-posedness of \eqref{1.1}. To this end, we make use of the approximation methods. Since the diffusion term coefficients are locally bounded and the noise is a Brownian sheet, our result is different from Gy\"{o}ngy \cite{MR1608641}. We establish $ L^{\rho}(\rho\geq2)$ estimates to ensure that the well-posedness of the solutions, but It\^{o} formula can not apply to  our situation.  In order to  overcome this difficulty, we first study a degenerate equation, and establish  that  there exists a family of degenerate equations  and their solution converge to  the solution  of non-degenerate equation  in $L^{\rho}$ space, thus we can use the uniform estimate to  the solution of degenerate equation to get the uniform  estimate  for the solution of non-degenerate equation. In addition,   we need the tightness  for stochastic convolution term,  so  we impose the condition  $\rho>6$.

Instead of using the time discretization method to derive the large deviation principle, we use a similar argument in \cite{MR3575422,MR2435853} to $C([0,T];L^{\rho}([0,1]))$-valued solutions. For the time discretization method, one of Cardon-Weber's  assumption \cite{MR1720097} is that the diffusion term coefficients are bounded, author can establish the maximum norm estimate of  $Z(h)$  to control the continuity of  the skeleton equations, and then the Freidlin-Wentzell inequality can be obtained. In view of the diffusion term coefficients are locally bounded, we can not obtain the maximum norm estimate for the solution $Z(h)$ of the skeleton equation as \cite{MR1486551} by using the continuity of the diffusion term $\sigma$. Thus we adopt the weak convergence method to obtain large deviation principle. In addition, compared with \cite{MR3575422,MR2435853}, when we deal with the controlled equation converge to skeleton equation, the solutions to two class of equations do not lie in $C([0,T]\times[0,1])$, we have to obtain the convergence in solution space $C([0,T];L^{\rho}([0,1]))$.

This paper is organized as follows: In Section 2,  we review some technical lemmas as well as the underlying assumptions. In Section 3, we obtain the  existence and uniqueness of the global solution. In Section 4, we give the definition of the large deviation principle and state a sufficient condition to ensure the existence of large deviation principle. In Section 5, we introduce the controlled equation and skeleton equation.  In Section 6, we respectively verify the conditions stated in Section 4 and prove the large deviation principle of the equation.

\section{Preliminaries}
In this Section, some basic assumptions and preliminaries are given to construct the global solution of  \eqref{1.1}:
\begin{description}
  \item[(H1)] $\eta\in L^{\rho}$ and $\rho>6$.
  \item[(H2)] The function $\sigma(t,x,r)$  is uniformly continuous for variable $t,x,r$ and $\sigma$ is globally Lipschitz with Lipschitz constant $L_{\sigma}$  and $|\sigma(t,x,r)|\leq K(1+|r|) $ .
  \item[(H3)] $f$ and $g$ are locally Lipschitz for the third variable, i.e. there exists  a
   constant $L$ such that for $(t,x,p,q)\in[0,T]\times[0,1]\times\mathbb{R}^{2}$,
  \begin{align}\label{2.1}
  |f(t,x,p)-f(t,x,q)|+|g(t,x,p)-g(t,x,q)|\leq L(1+|p|+|q|)|p-q|.
  \end{align}
  \item[(H4)] The function $g$ has the form
  \begin{align}\label{2.2}
  g(t,x,r)=g_{1}(t,x,r)+g_{2}(t,r),
  \end{align}
  where $g_{1}$ and $g_{2}$ are the Borel functions such that there exist a constant $K>0$ for which  for any $(t,x,r)\in [0,T]\times [0,1]\times\mathbb{R}$,
  \begin{align}\label{2.3}
  |g_{1}(t,x,r)|\leq K(1+|r|), \quad |g_{2}(t,r)|\leq K(1+|r|^{2}).
  \end{align}
  \item[(H5)] The function $f$ satisfies linear growth condition: there exists a $K>0$ such that for any $(t,x,r)\in [0,T]\times [0,1]\times\mathbb{R}$ we have
  \begin{align}\label{2.4}
  |f(t,x,r)|\leq K(1+|r|).
  \end{align}
\end{description}
 As \cite{MR1608641,MR1720097,MR3575422}, we impose the same conditions (\textbf{H2})-(\textbf{H5}). For a class of semi-linear parabolic equations, we need to add the condition (\textbf{H1}) to obtain the global solutions. We denote the norm of $L^{\rho}([0,1])$ as $\|\cdot\|_{\rho}$, and $C>0$ may be different from line to line.
\begin{definiton}[Mild solution]\label{mild solution}
We say stochastic process $u^{\epsilon}=\{u^{\epsilon}(t,x):t\in[0,T], x\in[0,1]\}$ is the solution of \eqref{1.1} if the $u^{\epsilon}(t,x)$ is $\mathcal{F}_{t}$-adapted $L^{\rho}([0,1])$-valued continuous  solution and satisfy the formulation
\begin{equation}
\begin{aligned}
&u^{\epsilon}(t,x)=\int_{0}^{1}G_{t}(x,y)\eta(y)dy+\int_{0}^{t}\int_{0}^{1}G_{t-s}(x,y)f(s,y,u^{\epsilon}(s,y))dyds\nonumber\\
&-\int_{0}^{t}\int_{0}^{1}\partial_{y}G_{t-s}(x,y)g(s,y,u^{\epsilon}(s,y))dyds+\sqrt{\epsilon}\int_{0}^{t}\int_{0}^{1}G_{t-s}(x,y)\sigma(s,y,u^{\epsilon}(s,y))W(dy,ds)\nonumber.
\end{aligned}
\end{equation}
\end{definiton}
\begin{remark}
It is worth noting that the last integral of the above formula should be understood in the meaning of It\^{o} integral sense and $W(dy,ds)=\sum_{k}h_k(y)dydw^{k}_s$, and $G_{t}(x,y)$ is the Green function of the heat operator $\frac{\partial}{\partial t}-\frac{\partial^{2}}{\partial x^{2}}$ with Dirichlet's boundary condition. It has the explicit form
\begin{align}\nonumber
G_{t}(x, y)=\frac{1}{\sqrt{4 \pi t}} \sum_{n=-\infty}^{n=+\infty}\left[\exp \left\{\frac{-(y-x-2 n)^{2}}{4 t}\right\}-\exp \left\{\frac{-(y+x-2 n)^{2}}{4 t}\right\}\right]
\end{align}
or
\begin{align}\nonumber
G_{t}(x, y)=1_{\{t>0\}} \sum_{k=1}^{\infty} \mathrm{e}^{\lambda_{k} t} \phi_{k}(x) \phi_{k}(y),
\end{align}
where $t\in [0,T], x\in [0,1],y\in[0,1]$. For $k=1,2,\cdots$
\begin{equation}\nonumber
\begin{cases}\phi_{k}(x)=\sqrt{2} \sin (2 k \pi x), &   \\ \lambda_{k}=4  k^{2} \pi^{2}. & \end{cases}
\end{equation}
The function $\{\phi_{k}(x)\}_{k\in\mathbb{N}}$  are an orthonormal basis of $L^{2}([0,1])$ consisted of the eigenfunction corresponding to the eigenvalue of the operator $\frac{\partial^{2}}{\partial x^{2}}$. $G_{t}(x,y)$  has the following useful estimates
\begin{align}\label{2.5}
\left|G_{t}(x,y)\right|\leq \frac{K_{1}}{t^{\frac{1}{2}}}\exp\left\{-a\frac{(x-y)^{2}}{t} \right\},
\end{align}
\begin{align}\label{2.6}
\left|\frac{\partial G}{\partial y}(t,x,y)\right|\leq \frac{K_{1}}{t}\exp\left\{-b\frac{(x-y)^{2}}{t}\right\},
\end{align}
\begin{align}\label{2.7}
\left|\frac{\partial G}{\partial x\partial y}(t,x,y)\right|\leq \frac{K_{1}}{t^{\frac{3}{2}}}\exp\left\{-c\frac{(x-y)^{2}}{t}\right\}.
\end{align}
The following result is  from Lemma 3.1 of  \cite{MR1608641} and the appendix of Chenal and Millet \cite{MR1486551}.
\begin{lemma}\label{lemma 2.1}
Let $J$ be  a linear operator defined for $v\in L^{\infty}([0,T],L^{1}([0,1]))$, $t\in[0,T]$, and $x\in [0,1]$ by
$$J(v)(t,x)=\int_{0}^{t}\int_{0}^{1}H(r,t,x,y)v(r,y)dydr$$
with $H(t,s,x,y)=\partial_{y}G_{t-s}(x,y), \frac{\partial G}{\partial x\partial y}(t,x,y)$. Then for any $\rho\in[1,\infty], q\in [1,\rho), q<\infty$ such that $k=1+\frac{1}{\rho}-\frac{1}{q}$, $J$ is a bounded linear operator of $L^{\gamma}([0,T];L^{q})$ in $C([0,T];L^{\rho}([0,1]))$ for $\gamma >2k^{-1}$, Moreover $J$ satisfies the following inequality:\\
For any $T\geq0$, for every $t\leq T$ and $\gamma>2k^{-1}$,
\begin{align}
\|J(v)(t,\cdot)\|_{\rho}\leq C\int_{0}^{t}(t-r)^{\frac{1}{2}k-1}\|v(r,\cdot)\|_{q}dr\leq C t^{\frac{1}{2}k-\frac{1}{\gamma}}\left(\int_{0}^{t}\|v(r,\cdot)\|^{\gamma}_{q}dr\right)^{\frac{1}{\gamma}}.
\end{align}
For $T\geq t>s\geq 0$, $0<\alpha<\frac{1}{2}\kappa$, for any $\gamma>(\frac{1}{2}\kappa-\alpha)^{-1}$, there exists a constant $C$ such that
\begin{equation}
\begin{aligned}
\|J(v)(t,\cdot)-J(v)(s,\cdot)\|_{\rho}\leq C|t-s|^{\alpha}\left(\int_{0}^{s}\|v(r,\cdot)\|_{q}^{\gamma}dr\right)^{\frac{1}{\gamma}}.
\end{aligned}
\end{equation}
For each $T>0$, $\beta\in(0,\kappa)$, for any $\delta>2(\kappa-\beta)^{-1}$, there is a constant $C$ such that
\begin{equation}
\begin{aligned}
\|J(v)(t,\cdot)-J(v)(t,\cdot+z)\|_{\rho}\leq C|z|^{\beta}\left(\int_{0}^{t}\|v(r,\cdot)\|^{\delta}_{q}\right)^{\frac{1}{\delta}},
\end{aligned}
\end{equation}
where $t\in[0,T]$ and $z\in[0,1]$.\\
For all $x,y\in [0,1], 0<s\leq t\leq T$, we have
\begin{equation}
\begin{aligned}
\sup_{t\in[0,T]}\int_{0}^{t}\int_{0}^{1}|G_{u}(x,z)-G_{u}(y,z)|^{p}dzdu\leq C|x-y|^{3-p},\quad \frac{3}{2}<p<3,
\end{aligned}
\end{equation}
\begin{equation}
\begin{aligned}
\sup_{x\in[0,1]}\int_{0}^{s}\int_{0}^{1}|G_{t-u}(x,z)-G_{s-u}(y,z)|^{p}dzdu\leq C|t-s|^{\frac{3-p}{2}},\quad 1<p<3,
\end{aligned}
\end{equation}
\begin{equation}
\begin{aligned}
\sup_{x\in[0,1]}\int_{s}^{t}\int_{0}^{1}|G_{u}(x,z)|^{p}dzdu\leq C|t-s|^{\frac{3-p}{2}},\quad 1<p<3.
\end{aligned}
\end{equation}
\end{lemma}
More properties of Green function $G_{t}(x,y)$ can be found in  \cite{MR1486551,MR2914776}.
\end{remark}
\section{The existence and uniqueness}

Our objection in this Section is to construct the $\mathcal{F}_t$-adapted $L^{\rho}([0,1])$-valued continuous global solution, it is equipped with a norm
$$\|u\|=(\max_{t\in[0,T]}\mathbb{E}\|u(t)\|^{\rho}_{\rho})^{\frac{1}{\rho}}.$$
We consider another space $B_{\delta,T}$, the definition of $B_{\delta,T}$ as follows
\begin{equation}\nonumber
B_{\delta,T}:=\{u(t,x):\max_{0\leq t\leq T}e^{-\rho\delta t}\mathbb{E}\|u(t)\|^{\rho}_{\rho}<\infty\},
\end{equation}
and equipped norm by
\begin{align}\nonumber
\|u\|_{B_{\delta,T}}=\left(\max_{0\leq t\leq T}e^{-\rho\delta t}\mathbb{E}\|u(t)\|^{\rho}_{\rho}\right)^{\frac{1}{\rho}},
\end{align}
where $\delta$ is a positive constant and will be determined later. It is obviously that $B_{\delta,T}$ is equivalent to $\mathcal{F}_t$-adapted $L^{\rho}([0,1])$-valued continuous function space. Hence we only construct the solution in $B_{\delta,T}$. In order to get the well-posed of the solutions, we first construct the solutions of the truncated equations and then by the compactness argument to guarantee the existence and uniqueness of the original equation \eqref{1.1}. Now consider the following truncated equation
\begin{equation}\label{3.1*}
\begin{aligned}
\frac{\partial u^{\epsilon}}{\partial t}(t,x)=&\frac{\partial^{2} u^{\epsilon}}{\partial x^{2}}(t,x)+\frac{\partial }{\partial x}\chi_{R}(\|u^{\epsilon}(t)\|_{\rho})g(t,x, u^{\epsilon}(t,x))+\chi_{R}(\|u^{\epsilon}(t)\|_{\rho})f(t,x,u^{\epsilon}(t,x))\\&+\sqrt{\epsilon}\chi_{R}(\|u^{\epsilon}(t)\|_{\rho})\sigma(t,x,u^{\epsilon}(t,x))\frac{\partial^{2} W}{\partial t \partial x}(t,x), \\
\end{aligned}
\end{equation}
with Dirichlet boundary condition $u^{\epsilon}(t,0)=u^{\epsilon}(t,1)=0$ and initial condition $\eta_{R}^{\epsilon}(0,x)=\eta(x)$ for $t\in [0,T], x\in[0,1]$, where $\chi_{R}(r)$ is a first-order differentiable function and its derivative is bounded, it satisfies
\begin{equation}
\left\{\begin{aligned}
&\chi_{R}(r)=1,\quad |r|\leq R,\\
&\chi_{R}(r)=0, \quad |r|\geq R+1.
\end{aligned}\right.
\end{equation}
\begin{theorem}\label{Theorem 3.1}
Assume that $f$, $g$, $\sigma$ satisfy conditions \textbf{(H2)}-\textbf{(H5)} and $\eta(x)\in L^{\rho}([0,1]),\rho>6$, then there exists a unique $L^{\rho}([0,1])$-valued continuous solution for the truncated equation \eqref{3.1*}.
\end{theorem}
\begin{proof}
The proof of this Theorem mainly applies the Banach fixed point argument. According to Definition \ref{mild solution}, the solution is expressed as follows
\begin{align}\label{3.1}
u^{\epsilon}(t,x):=A^{\epsilon}_{1}(t,x)+A^{\epsilon}_{2}(t,x)+A^{\epsilon}_{3}(t,x)+A^{\epsilon}_{4}(t,x),
\end{align}
where
\begin{equation}\nonumber
  \begin{aligned}
A^{\epsilon}_{1}(t,x)&=\int_{0}^{1}G_{t}(x,y)\eta(y)dy,\\
A^{\epsilon}_{2}(t,x)&=\int_{0}^{t}\int_{0}^{1}G_{t-s}(x,y)\chi_{R}(\|u^{\epsilon}(t)\|_{\rho})f(s,y,u^{\epsilon}(s,y))dyds,\\
A^{\epsilon}_{3}(t,x)&=-\int_{0}^{t}\int_{0}^{1}\partial_{y}G_{t-s}(x,y)\chi_{R}(\|u^{\epsilon}(t)\|_{\rho})g(s,y,u^{\epsilon}(s,y))dyds,\\
A^{\epsilon}_{4}(t,x)&=\sqrt{\epsilon}\int_{0}^{t}\int_{0}^{1}G_{t-s}(x,y)\chi_{R}(\|u^{\epsilon}(t)\|_{\rho})\sigma(s,y,u^{\epsilon}(s,y))W(dy,ds).
     \end{aligned}
\end{equation}
Furthermore, for fixed $R>0$, we consider the map $\mathcal{M}_{T}$ of the solution on $B_{\delta,T}$
\begin{align}\label{3.2}
\mathcal{M}_{T}(u^{\epsilon}(t,x))=A^{\epsilon}_{1}(t,x)+A^{\epsilon}_{2}(t,x)+A^{\epsilon}_{3}(t,x)+A^{\epsilon}_{4}(t,x).
\end{align}
Hence, \eqref{3.1} and \eqref{3.2} imply that we need to prove that there is a unique fixed point for the mapping $\mathcal{M}_{T}$, then we mainly check the invariance and contraction of $\mathcal{M}_{T}$ in the next discussion.\\
\textbf{Invariance:} For any $u^{\epsilon}\in B_{\delta,T}$, checking $\mathcal{M}_{T}(u^{\epsilon})\in B_{\delta,T}$. For $A_{1}^{\epsilon}(t)$,  with the help of Young's inequality and \eqref{2.5}, we have
\begin{equation}\label{3.3}
E\|A_{1}^{\epsilon}(t)\|_{\rho}^{\rho}=\int_{0}^{1}\left(\int_{0}^{1}G_{t}(x,y)\eta(y)dy\right)^{\rho}dx\leq C \|\eta\|_{\rho}^{\rho}.
\end{equation}
Owing to the assumption $(\mathbf{H}5)$, Minkowski's  inequality, Young's inequality and \eqref{2.5}, we figure out
\begin{equation}\label{3.4}
\begin{aligned}
E\|A_{2}^{\epsilon}(t)\|_{\rho}^{\rho}&=\mathbb{E}\int_{0}^{1}\left(\int_{0}^{t}\int_{0}^{1}G_{t-s}(x,y)\chi_{R}(\|u^{\epsilon}(s)\|_{\rho})f(s,y,u^{\epsilon}(s,y))dyds\right)^{\rho}dx\\
&\leq \mathbb{E}\left[\int_{0}^{t}\left(\int_{0}^{1}\left[\int_{0}^{1}G_{t-s}(x,y)\chi_{R}(\|u^{\epsilon}(s)\|_{\rho})f(s,y,u^{\epsilon}(s,y))dy\right]^{\rho}dx\right)^{\frac{1}{\rho}} ds\right]^{\rho}\\
&\leq \mathbb{E}\left[\int_{0}^{t}\|G_{t-s}(\cdot)\|_{1}\|\chi_{R}(\|u^{\epsilon}(s)\|_{\rho})f(s,\cdot,u^{\epsilon}(s,\cdot))\|_{\rho}ds\right]^{\rho}\\
&\leq C_{R}t^{\rho}.
\end{aligned}
\end{equation}
Similarly, using assumption $(\mathbf{H}4)$, Minkowski's inequality, Young's inequality and \eqref{2.6} we have
\begin{equation}\label{3.5}
\begin{aligned}
\mathbb{E}\|A_{3}^{\epsilon}(t)\|_{\rho}^{\rho}&=\mathbb{E}\int_{0}^{1}\left(\int_{0}^{t}\int_{0}^{1}\partial_{y}G_{t-s}(x,y)\chi_{R}(\|u^{\epsilon}(s)\|_{\rho})g(s,y,u^{\epsilon}(s,y))dyds\right)^{\rho}dx\\
&\leq\mathbb{E}\left[\int_{0}^{t}\left(\int_{0}^{1}\left(\int_{0}^{1}\partial_{y}G_{t-s}(x,y)\chi_{R}(\|u^{\epsilon}(s)\|_{\rho})g(s,y,u^{\epsilon}(s,y))dy\right)^{\rho}dx\right)^{\frac{1}{\rho}}ds\right]^{\rho}\\
&\leq\mathbb{E}\left[\int_{0}^{t}\|\partial_{y}G_{t-s}(\cdot)\|_{\rho}\chi_{R}(\|u^{\epsilon}(s)\|_{\rho})\|g(s,\cdot,u^{\epsilon}(s,\cdot))\|_{1}ds\right]^{\rho}\\
&\leq C\mathbb{E}\left[\int_{0}^{t}(t-s)^{-1+\frac{1}{2\rho}}\chi_{R}(\|u^{\epsilon}(s)\|_{\rho})\|g(s,\cdot,u^{\epsilon}(s,\cdot))\|_{1}ds\right]^{\rho}\\
&\leq C_{R}t^{\frac{1}{2}}.
\end{aligned}
\end{equation}
Finally, for $A_{4}^{\epsilon}(t,x)$ using Fubini's theorem, B-D-G inequality, H\"{o}lder's inequality, the assumption \textbf{(H2)},  Young's inequality and Minkovski's inequality to get
\begin{equation}\label{3.6}
\begin{aligned}
\mathbb{E}\|A_{4}^{\epsilon}(t,x)\|_{\rho}^{\rho}&=\epsilon^{\frac{\rho}{2}}\mathbb{E}\int_{0}^{1}\left[\int_{0}^{t}\int_{0}^{1}G_{t-s}(x,y)\chi_{R}(\|u^{\epsilon}(s)\|_{\rho})\sigma(s,y,u^{\epsilon}(s,y))W(dy,ds)\right]^{\rho}dx\\
&=\epsilon^{\frac{\rho}{2}}\mathbb{E}\int_{0}^{1}\left[\int_{0}^{t}\int_{0}^{1}G_{t-s}(x,y)\chi_{R}(\|u^{\epsilon}(s)\|_{\rho})\sigma(s,y,u^{\epsilon}(s,y))\sum_{k}h_k(y)dydw_s^k\right]^{\rho}dx\\
&\leq\epsilon^{\frac{\rho}{2}}\mathbb{E}\int_{0}^{1}\left[\int_{0}^{t}\int_{0}^{1}G^{2}_{t-s}(x,y)\chi_{R}(\|u^{\epsilon}(s)\|_{\rho})\sigma^{2}(s,y,u^{\epsilon}(s,y))ydyds\right]^{\frac{\rho}{2}}dx\\
&=C\epsilon^{\frac{\rho}{2}}\mathbb{E}\left[\int_{0}^{t}(t-s)^{-\frac{1}{2}}\chi_{R}(\|u^{\epsilon}(s)\|_{\rho})\|u^{\epsilon}(s)\|^{2}_{\rho}ds\right]^{\frac{\rho}{2}}\\
&\leq C_{R}\epsilon^{\frac{\rho}{2}}t^{\frac{\rho}{4}}.
\end{aligned}
\end{equation}
Thus, by \eqref{3.2}-\eqref{3.6}, we have
$$
\|\mathcal{M}_{T}(u^{\epsilon})\|_{B_{\delta,T,\rho}}\leq C,
$$
where C is only depend on  $T$, $R$ and $\epsilon$.\\
\textbf{Contraction:} For any $u^{\epsilon}(t,x),v^{\epsilon}(t,x)\in B_{T,\delta}$, we need to estimate $\mathcal{M}_{T}\left(u^{\epsilon}(t, x)\right)-\mathcal{M}_{T}\left(v^{\epsilon}(t, x)\right)$,
\begin{equation}
\begin{aligned}
\mathcal{M}_{T}\left(u^{\epsilon}(t, x)\right)-\mathcal{M}_{T}\left(v^{\epsilon}(t, x)\right):=B_{1}^{\epsilon}(t,x)+B_{2}^{\epsilon}(t,x)+B_{3}^{\epsilon}(t,x),
\end{aligned}
\end{equation}
where
\begin{equation}\nonumber
  \begin{aligned}
B_{1}^{\epsilon}(t,x)&=\int_{0}^{t}\int_{0}^{1} G_{t-s}(x,y)\biggr(\chi_{R}(\|u^{\epsilon}(s)\|_{\rho})f(s,y,u^{\epsilon}(s,y))-\chi_{R}(\|v^{\epsilon}(s)\|_{\rho})f(s,y,v^{\epsilon}(s,y))\biggr)dyds,\\
B_{2}^{\epsilon}(t,x)&=-\int_{0}^{t}\int_{0}^{1}\partial_{y}G_{t-s}(x,y)\biggr(\chi_{R}(\|u^{\epsilon}(s)\|_{\rho})g(s,y,u^{\epsilon}(s,y))-\chi_{R}(\|v^{\epsilon}(s)\|_{\rho})g(s,y,v^{\epsilon}(s,y))\biggr)dyds,\\
B_{3}^{\epsilon}(t,x)&=\sqrt{\epsilon}\int_{0}^{t}\int_{0}^{1}G_{t-s}(x,y)\biggr(\chi_{R}(\|u^{\epsilon}(s)\|_{\rho})\sigma(s,y,u^{\epsilon}(s,y))-\chi_{R}(\|v^{\epsilon}(s)\|_{\rho})\sigma(s,y,v^{\epsilon}(s,y))\biggr)W(dy,ds).
    \end{aligned}
\end{equation}
For  $B_{1}^{\epsilon}(t,x)$, either $\|u^\epsilon(s)\|_{\rho}\leq\|v^\epsilon(s)\|_{\rho} $ or $\|u^\epsilon(s)\|_{\rho}\geq\|v^\epsilon(s)\|_{\rho} $, for any $s\in[0,T]$. Without loss of generality, let  $\|u^\epsilon(s)\|_{\rho}\leq\|v^\epsilon(s)\|_{\rho}$,  then
\begin{align*}
  &\left\|\chi_{R}(\|u^{\epsilon}(s)\|_{\rho})f\left(s, \cdot, u^{\epsilon}(s, \cdot)\right)-\chi_{R}(\|v^{\epsilon}(s)\|_{\rho})f\left(s, \cdot, v^{\epsilon}(s, \cdot)\right)\right\|_{1}\\
  &~~~~~\leq \left\|(\chi_{R}(\|u^{\epsilon}(s)\|_{\rho})-\chi_{R}(\|v^{\epsilon}(s)\|_{\rho}))f(s,y,u^{\epsilon}(s,y))\right\|_{1}\\
  &~~~~~~~~~+\left\|\chi_{R}(\|v^{\epsilon}(s)\|_{\rho})(f(s,y,u^{\epsilon}(s,y))-f(s,y,u^{\epsilon}(s,y)))\right\|_{1}\\
  &\leq \|u^{\epsilon}(s)-v^{\epsilon}(s)\|_{\rho}(1+R)+(2R+1)\|u^{\epsilon}(s)-v^{\epsilon}(s)\|_{\rho}.
\end{align*}
 By assumption \textbf{(H3)}, Minkowski's inequality, Young's inequality and H\"{o}lder's inequality we have
\begin{equation}\label{3.8}
\begin{small}
\begin{aligned}
\mathbb{E}\|B_{1}^{\epsilon}(t,\cdot)\|_{\rho}^{\rho}=&\mathbb{E} \int_{0}^{1}\left[\int_{0}^{t}\int_{0}^{1} G_{t-s}(x,y)\left(\chi_{R}(\|u^{\epsilon}(s)\|_{\rho})f(s,y,u^{\epsilon}(s,y))-\chi_{R}(\|v^{\epsilon}(s)\|_{\rho})f(s,y,v^{\epsilon}(s,y))\right)dyds \right]^{\rho}dx\\
\leq &\mathbb{E}\left\{\int_{0}^{t}\left[\int_{0}^{1}\left(\int_{0}^{1}G_{t-s}(x,y)(\chi_{R}(\|u^{\epsilon}(s)\|_{\rho})f\left(s, y, u^{\epsilon}(s, y)\right)-\chi_{R}(\|v^{\epsilon}(s)\|_{\rho})f\left(s, y, v^{\epsilon}(s, y)\right))dy\right)^{\rho}dx\right]^{\frac{1}{\rho}}ds\right\}^{\rho}\\\leq& \mathbb{E}\left\{\int_{0}^{t}\|G_{t-s}(\cdot)\|_{\rho}\bigg\|\chi_{R}(\|u^{\epsilon}(s)\|_{\rho})f\left(s, \cdot, u^{\epsilon}(s, \cdot)\right)-\chi_{R}(\|v^{\epsilon}(s)\|_{\rho})f\left(s, \cdot, v^{\epsilon}(s, \cdot)\right)\right\|_{1}ds\bigg\}^{\rho}\\
\leq& C\mathbb{E}\left\{\int_{0}^{t}(t-s)^{-\frac{1}{2}+\frac{1}{2\rho}}(3R+2)\|u^{\epsilon}(s)-v^{\epsilon}(s)\|_{\rho}ds\right\}^{\rho}\\
\leq& C_{R}\mathbb{E}\left\{\int_{0}^{t}(t-s)^{-\frac{1}{2}+\frac{1}{2\rho}}\|u^{\epsilon}(s)-v^{\epsilon}(s)\|_{\rho}ds\right\}^{\rho}.\\
\end{aligned}
\end{small}
\end{equation}
Furthermore, using \eqref{3.8} we have
\begin{equation}\label{3.9}
\begin{aligned}
\|B_{1}^{\epsilon}\|^{\rho}_{B_{T,\delta,\rho}}\leq & C_{R}\left\{\int_{0}^{t}e^{-\delta(t-s)}(t-s)^{-\frac{1}{2}+\frac{1}{2\rho}}\|u^{\epsilon}(s)-v^{\epsilon}(s)\|_{B_{T,\delta,\rho}}ds\right\}^{\rho}\\
\leq &C_{R}\left(\frac{1}{\delta}\right)^{\frac{\rho+1}{2}}\Gamma^{\rho}(\frac{\rho+1}{2\rho})\|u-v\|_{B_{T,\delta,\rho}}^{\rho}.
\end{aligned}
\end{equation}
And then, using the assumption \textbf{(H3)}, \eqref{2.6} and the similar calculations as before, we obtain
\begin{equation}\label{3.10}
\begin{aligned}
\mathbb{E}\|B_{2}^{\epsilon}(t,\cdot)\|^{\rho}_{\rho}&=\mathbb{E}\int_{0}^{1}\left[\int_{0}^{t}\int_{0}^{1}\partial_{y}G_{t-s}(x,y)\bigg(g(s,y,u^{\epsilon}(s,y))-g(s,y,v^{\epsilon}(s,y))\bigg)dyds\right]^{\rho}dx\\
&\leq C_{R}\mathbb{E}\left(\int_{0}^{t}(t-s)^{-1+\frac{1}{2\rho}}\|u^{\epsilon}(s)-v^{\epsilon}(s)\|_{\rho}ds\right)^{\rho}.
\end{aligned}
\end{equation}
Similar to \eqref{3.9} we have
\begin{equation}\label{3.11}
\begin{aligned}
  \left\|B_{2}^{\epsilon}\right\|_{B_{T, \delta,\rho}}^{\rho} \leq C_{R}\left(\frac{1}{\delta}\right)^{\frac{1}{2}}\Gamma^{\rho}(\frac{1}{2\rho})\|u^{\epsilon}-v\|_{B_{T, \delta,\rho}}^{\rho}.
\end{aligned}
\end{equation}
Finally, applying the assumption \textbf{(H2)}, Fubini's theorem, the B-D-G inequality, H\"{o}lder's inequality and similar calculations as before, one gets
\begin{equation}
\begin{aligned}
\mathbb{E}\left\|B_{3}^{\epsilon}(t, \cdot)\right\|_{\rho}^{\rho} &\leq C_{\rho} \epsilon^{\rho} \mathbb{E} \int_{0}^{1}\left(\int _ { 0 } ^ { t } \int _ { 0 } ^ { 1 } G _ { t - s } ^ { 2 } ( x , y ) \left(\chi_{R}\left(\left\|u^{\epsilon}(s)\right\|_{\rho}\right) \sigma\left(s, y, u^{\epsilon}(s, y)\right)\right.\right.\\
&\left.\left.\quad-\chi_{R}\left(\left\|v^{\epsilon}(s)\right\|_{\rho}\right) \sigma\left(s, y, v^{\epsilon}(s, y)\right)\right)^{2} y d y d s\right)^{\frac{\rho}{2}} d x \\
& \leq C_{R,\epsilon,\rho}\left\{\int_{0}^{t}(t-s)^{-1+\frac{1}{\rho}}\left(\mathbb{E}\left\|u^{\epsilon}(s)-v^{\epsilon}(s)\right\|_{\rho}^{\rho}\right)^{\frac{\rho}{2}} d s\right\}^{\frac{\rho}{2}} \\
& \leq C_{R,\epsilon,\rho}\left\{\int_{0}^{t}(t-s)^{-1+\frac{1}{\rho}} e^{2 \delta s} d s\right\}^{\frac{\rho}{2}}\left\|u^{\epsilon}-v^{\epsilon}\right\|_{B_{T, \delta, \rho}}^{\rho}.
\end{aligned}
\end{equation}
Hence,
\begin{equation}\label{3.13}
\begin{aligned}
\left\|B_{3}^{\epsilon}\right\|_{B_{T, \delta},\rho}^{\rho}\leq C_{R,\epsilon,\rho}\left(\frac{1}{\delta}\right)^{\frac{1}{2}}\Gamma^{\frac{\rho}{2}}(\frac{1}{\rho})\|u-v\|_{B_{T, \delta,\rho}}^{\rho}.
\end{aligned}
\end{equation}
\eqref{3.9}, \eqref{3.11} and \eqref{3.13} imply that we can choose a large enough $\delta>0$ such that $C_{R}\left(\frac{1}{\delta}\right)^{\frac{\rho+1}{2}}\Gamma^{\rho}(\frac{\rho+1}{2\rho})+C_{R}\left(\frac{1}{\delta}\right)^{\frac{1}{2}}\Gamma^{\rho}(\frac{1}{2\rho})+\epsilon C_{R}\left(\frac{1}{\delta}\right)^{\frac{1}{2}}\Gamma^{\frac{\rho}{2}}(\frac{1}{\rho})<1$, we get the contraction of map $\mathcal{M}_{T}$. Then there exists a unique solution for equation \eqref{1.1} by the Banach fixed point theorem.
\end{proof}
Based on the well-posedness of the above local solution,  if equation \eqref{1.1} has a global solution, then it must be unique. Then we proceed to establish the global solution of \eqref{1.1}. To this end, we will use the compactness arguments to obtain the global solution. So let us consider the following approximate equation of \eqref{1.1}
\begin{equation}\label{3.16}
\left\{\begin{aligned}
\frac{\partial u_{n}^{\epsilon}}{\partial t}(t,x)&=\frac{\partial^{2} u_{n}^{\epsilon}}{\partial x^{2}}(t,x)+\frac{\partial }{\partial x}g_{n}(t,x, u_{n}^{\epsilon}(t,x))+f_{n}(t,x,u_{n}^{\epsilon}(t,x))+\sqrt{\epsilon}\sigma_{n}(t,x,u_{n}^{\epsilon}(t,x))\frac{\partial^{2} W}{\partial t \partial x}(t,x), \\
u_{n}^{\epsilon}(t,0)&=u_{n}^{\epsilon}(t,1)=0, \\
u_{n}^{\epsilon}(0,x)&=\eta_{0n}(x),
\end{aligned}\right.
\end{equation}
where   we take sequences of  bounded Borel functions $f_{n}:=f_{n}(t,x,r), g_{n}:=g_{n}(t,x,r)$ and $\sigma_{n}:=\sigma_{n}(t,x,r)$ such that they are globally Lipshitz in $r\in\mathbb{R}$  and
\begin{itemize}
  \item $f_{n}(t,x,r)=f(t,x,r),g_{n}(t,x,r)=g(t,x,r),\sigma_{n}(t,x,r)=\sigma(t,x,r) \quad \text{for}\quad |r|\leq n$,
  \item $f_{n}(t,x,r)=0,g_{n}(t,x,r)=0,\sigma_{n}(t,x,r)=0 \quad \text{for}\quad |r|\geq n+1$,
\end{itemize}
 $f_{n}(t,x,r),g_{n}(t,x,r)=g_{1,n}(t,x,r)+g_{2,n}(t,r),\sigma_{n}(t,x,r)$ satisfy the same growth  and Lipschitz condition as $f(t,x,r),g(t,x,r),\sigma(t,x,r)$  and with constants which are independents of $n$. Consider a   bounded  and smooth   sequence $\{\eta_{0,n}\}_{n\in \mathbb{N}}$ converges to $\eta$ in $L^{\rho}$. Similar to the previous section, the equation \eqref{3.16} can be expressed in the form of a mild solution.
\begin{equation}\label{3.16*}
\begin{aligned}
u^{\epsilon}_{n}(t, x)=&\int_{0}^{1} G_{t}(x, y) \eta_{0,n}(y) d y+\int_{0}^{t} \int_{0}^{1} G_{t-s}(x, y) f_n\left(s, y, u^{\epsilon}_{n}(s, y)\right) d y d s \\
&-\int_{0}^{t} \int_{0}^{1} \partial_{y} G_{t-s}(x, y) g_n\left(s, y, u^{\epsilon}_{n}(s, y)\right) d y d s\\
&+\sqrt{\epsilon} \int_{0}^{t} \int_{0}^{1} G_{t-s}(x, y) \sigma_n\left(s, y, u^{\epsilon}_{n}(s, y)\right)W(d y, d s)\\
=&: J_{1,n}+J^{\epsilon}_{2,n}+J^{\epsilon}_{3,n}+J^{\epsilon}_{4,n}.
\end{aligned}
\end{equation}
 In view of Theorem \ref{Theorem 3.1}, equation \eqref{3.16} exists a unique solution. Our final objection is to obtain the global solution from the sequence $\{u^\epsilon_{n}\}_{n\in\mathbb{N}}$, thus the next deliberation is the argument of the tightness of sequence $\{u^\epsilon_{n}\}_{n\in\mathbb{N}}$. Before that, it is necessary to establish  uniform estimate of the solutions. To this end,   we first  consider a  degenerate equation
 \begin{equation}\label{Degenerate equation}
  \begin{aligned}
    \frac{\partial u_n^{k,\epsilon}}{\partial t} (t,x)&=\frac{\partial^2 u_n^{k,\epsilon}}{\partial x^2} (t,x)+f_{n}(t,x,u_n^{k,\epsilon}(t,x))+\frac{\partial }{\partial x}g_{n}(t,x, u_n^{k,\epsilon}(t,x))\\
    &~~~~+\sum_{i=1}^{k}\sigma_{n}(t,x, u_n^{k,\epsilon}(t,x))h_{i}(x)dw_{i}(t)
  \end{aligned}
 \end{equation}
and its mild formulation
 \begin{equation}\label{Degenerate equation1}
  \begin{aligned}
   u_n^{k,\epsilon}(t,x)=&\int_{0}^{1}G_{t}(x,y)\eta_{0,n}(y)dy+\int_{0}^{t}\int_{0}^{1}G_{t-s}(x,y)f_{n}(s,y,u_n^{k,\epsilon}(s,y))dyds\\
    &-\int_{0}^{t}\int_{0}^{1}\partial_{y}G_{t-s}(x,y)g_{n}(s,y,u_n^{k,\epsilon}(s,y))dyds\\
    &+\sum_{i=1}^{k}\sqrt{\epsilon}\int_{0}^{t}\int_{0}^{1}G_{t-s}(x,y)\sigma_{n}(s,y,u^{k,\epsilon}_{n}(s,y))h_i(y)dydw_{i}(s).
  \end{aligned}
 \end{equation}
Equation  \eqref{Degenerate equation}  can be regard as an  evolution equation driven by a  finite  dimensional  Brownian motion, thus  \cite[Lemma 4.3]{MR1787126} can be applied here.  We  get  the uniform estimate for $\mathbb{E}\sup_{t\in[0,T]}\|u_{n}^{\epsilon}(t)\|_{\rho}^{\rho}$ by  $\mathbb{E}\sup_{t\in[0,T]}\|u_{n}^{k,\epsilon}(t)\|_{\rho}^{\rho}$, if we can show that $\|u_{n}^{k,\epsilon}(t)\|_{\rho}^{\rho}$ converging to $\|u_{n}^{\epsilon}(t)\|_{\rho}^{\rho}$ almost surely for sufficiently  large  $k$.
\begin{lemma}\label{convergence}
  There exists a subsequence $\{k_{n^\prime}\}_{n^\prime\in \mathbb{Z}^{+}}\subset \{k\}_{k\in \mathbb{Z}^{+}}$  such that
  $$\|u^\epsilon_n(t)-u_{n}^{k_{n^\prime},\epsilon}(t)\|_{\rho}\rightarrow 0,\quad n^{\prime}\rightarrow\infty$$
 for any $n>0,\rho\geq 2$, where the convergence  take places almost surely in a full measure set $\hat{\Omega}^0.$
\end{lemma}
\begin{proof}
  Note that we need only to consider $n\in\mathbb{Z}^{+}$.  Let  \eqref{3.16*} minus  \eqref{Degenerate equation1},  we have
  \begin{equation}
    \begin{aligned}
      u_{n}^\epsilon(t,x)-u_n^{k,\epsilon}(t,x)=&\int_{0}^{t}\int_{0}^{1}G_{t-s}(x,y)\left(f_{n}(s,y,u^\epsilon_n(s,y))-f_{n}(s,y,u^{k,\epsilon}_n(s,y)\right)dyds\\
      &-\int_{0}^{t}\int_{0}^{1}\partial_{y}G(t-s)(x,y)\left(g_{n}(s,y,u^\epsilon_{n}(s,y))-g_{n}(s,y,u_{n}^{k,\epsilon}(s,y))\right)dyds\\
      &+\sum_{i=1}^{k}\sqrt{\epsilon}\int_{0}^{t}\int_{0}^{1}G_{t-s}(x,y)\left(\sigma_{n}(s,y,u^{\epsilon}_{n}(s,y))-\sigma_{n}(s,y,u^{k,\epsilon}_{n}(s,y))\right)h_i(y)dydw_{i}(s)\\
      &+\sum_{i=k+1}^{\infty}\sqrt{\epsilon}\int_{0}^{t}\int_{0}^{1}G_{t-s}(x,y)\sigma_{n}(s,y,u^\epsilon_{n}(s,y))h_{i}(y)dydw_{i}(s)\\
      =:&\Delta_{1}(t,x)+\Delta_{2}(t,x)+\Delta_{3}(t,x)+\Delta_{4}(t,x).
    \end{aligned}
  \end{equation}
Thus,
\begin{equation}\label{4}
  \begin{aligned}
    \mathbb{E}\|u(t)-u_n^{k,\epsilon}(t)\|_{\rho}^{\rho}\leq \mathbb{E}\|\Delta_{1}(t)\|_{\rho}^{\rho}+\mathbb{E}\|\Delta_{2}(t)\|_{\rho}^{\rho}+\mathbb{E}\|\Delta_{3}(t)\|_{\rho}^{\rho}+\mathbb{E}\|\Delta_{4}(t)\|_{\rho}^{\rho}.
  \end{aligned}
\end{equation}
By Minkowski's inequality, Young's inequality, H\"{o}lder's inequality, \eqref{2.5}  and  the property of  global Lipshitz  to $f_{n}$, we have
\begin{equation}\label{5}
  \begin{aligned}
    \mathbb{E}\left\|\Delta_1(t)\right\|^{\rho}_{\rho}&\leq\mathbb{E}\left[\int_{0}^{t}\left(\int_{0}^{1}\left(\int_{0}^{1}G_{t-s}(x-y)\left(f_{n}(s,y,u^\epsilon_{n}(s,y))-f_{n}(s,y,u_{n}^{k,\epsilon}(s,y))\right)dy\right)^{\rho}dx\right)^{\frac{1}{\rho}}ds\right]^{\rho}\\
    &\leq \mathbb{E}\left[\int_{0}^{t}\|G_{t-s}(\cdot)\|_{1}\|f(s,\cdot,u^\epsilon_n(s))-f_n(s,\cdot,u_{n}^{k,\epsilon}(s))\|_{\rho}ds\right]^{\rho}\\
    &\leq C(n,\rho,T)\int_{0}^{t}\mathbb{E}\|u_{n}(s)-u_{n}^{k,\epsilon}(s)\|_{\rho}^{\rho}ds.
  \end{aligned}
\end{equation}
For $\Delta_{2}(t,x)$,  By \eqref{2.6}, the property of global Lipshitz  to  $g_{n}$   and similar inequalities as above, we obtain
\begin{equation}\label{6}
  \begin{aligned}
    \mathbb{E}\|\Delta_{2}(t)\|_{\rho}^{\rho} &\leq\mathbb{E}\left[\int_{0}^{t}\left(\int_{0}^{1}\left(\int_{0}^{1}\partial_{y}G_{t-s}(x-y)\left(g_{n}(s,y,u^\epsilon_{n}(s,y))-g_{n}(s,y,u_{n}^{k,\epsilon}(s,y))\right)dy\right)^{\rho}dx\right)^{\frac{1}{\rho}}ds\right]^{\rho}\\
    &\leq \mathbb{E}\left[\int_{0}^{t}\|\partial_yG_{t-s}(\cdot)\|_{1}\|f(s,\cdot,u^\epsilon_n(s))-f_n(s,\cdot,u_{n}^{k,\epsilon}(s))\|_{\rho}ds\right]^{\rho}\\
    &\leq C(n,\rho) \mathbb{E}\left[\int_{0}^{t}(t-s)^{-\frac{1}{2}}\|u^\epsilon_{n}(s)-u_{n}^{k,\epsilon}(s)\|_{\rho}ds\right]^{\rho}\\
    &\leq C(n,\rho,T)\int_{0}^{t}(t-s)^{-\frac{1}{2}}\mathbb{E}\|u^\epsilon_{n}(s)-u_{n}^{k,\epsilon}(s)\|^{\rho}_{\rho}ds.
  \end{aligned}
\end{equation}
 Applying  B-D-G inequality,  Fubini's Theorem,  H\"{o}lder's inequality, Minkowski's inequality Young's inequality, \eqref{2.5} and the property of global Lipschitz  to $\sigma_n$ to  $\Delta_{3}(t,x)$, we get
 \begin{equation}\label{7}
  \begin{aligned}
    \mathbb{E}\|\Delta_3(t)\|_{\rho}^{\rho}&\leq\epsilon^{\frac{\rho}{2}} \mathbb{E}\int_{0}^{1}\left(\int_{0}^{t}\sum_{i=1}^{k}\left(\int_{0}^{1}G_{t-s}(x,y)\left(\sigma_{n}(s,y,u^{\epsilon}_{n}(s,y))-\sigma_{n}(s,y,u^{k,\epsilon}_{n}(s,y))\right)h_i(y)dy\right)^{2}ds\right)^{\frac{\rho}{2}}dx\\
    &\leq\epsilon^{\frac{\rho}{2}} \mathbb{E}\int_{0}^{1}\left(\int_{0}^{t}\int_{0}^{1}G^2_{t-s}(x,y)\left(\sigma_{n}(s,y,u^{\epsilon}_{n}(s,y))-\sigma_{n}(s,y,u^{k,\epsilon}_{n}(s,y))\right)^2dyds\right)^{\frac{\rho}{2}}dx\\
    &\leq\epsilon^{\frac{\rho}{2}} \mathbb{E}\int_{0}^{1}\left(\int_{0}^{t}\int_{0}^{1}G^2_{t-s}(x,y)\left(\sigma_{n}(s,y,u^{\epsilon}_{n}(s,y))-\sigma_{n}(s,y,u^{k,\epsilon}_{n}(s,y))\right)^2dyds\right)^{\frac{\rho}{2}}dx\\
    &\leq\epsilon^{\frac{\rho}{2}} \mathbb{E}\left[\int_{0}^{t}\|G_{t-s}(\cdot)\|_{2}^{2}\|\sigma_{n}(s,\cdot,u^{\epsilon}_{n}(s))-\sigma_{n}(s,\cdot,u^{k,\epsilon}_{n}(s))\|^{2}_{\rho}ds \right]^{\frac{\rho}{2}}\\
    &\leq  C(n,\rho,T,\epsilon)   \int_{0}^{t}(t-s)^{-\frac{1}{2}}\mathbb{E}\|u^{\epsilon}_{n}(s)-u_{n}^{k,\epsilon}(s)\|^{\rho}_{\rho}ds.
  \end{aligned}
 \end{equation}
For $\Delta_{4}(t,x)$,  by Fubini's theorem,  B-D-G inequality, H\"{o}lder's inequality,  the boundedness   of  $\sigma_{n}$
\begin{equation}\label{8}
  \begin{aligned}
    \mathbb{E}\|\Delta_{4}(t)\|_{\rho}^{\rho}&\leq \epsilon^{\frac{\rho}{2}}
    \mathbb{E}\int_{0}^{1}\left(\sum_{i=k+1}^{\infty}\int_{0}^{t}\left(\int_{0}^{1}G_{t-s}(x,y)
    \sigma_n(s,y,u^{\epsilon}_{n}(s,y))h_{i}(y)dy\right)^{2}ds\right)^{\frac{\rho}{2}}dx\\
    &\leq \epsilon^{\frac{\rho}{2}}   \mathbb{E} \int_{0}^{1}\left(\int_{0}^{t}\int_{0}^{1}G_{t-s}^{2}(x,y)\sigma_{n}^2(s,y,u^{\epsilon}_n(s,y))\sum_{i=k+1}^{\infty}h_{i}^2(y)dyds\right)^{\frac{\rho}{2}}dx\\
    &\leq C(n,\rho,T,\epsilon)\left(\sum_{i=k+1}^{\infty}h_{i}^{2}(1)\right)^{\frac{\rho}{2}}.
  \end{aligned}
\end{equation}
Together with  \eqref{4}-\eqref{8}, Gr\"{o}nwall's inequality implies that
\begin{equation}
  \begin{aligned}
    \mathbb{E}\|u_{n}^\epsilon(t)-u_n^{k,\epsilon}(t)\|_{\rho}^{\rho}\leq  C(n,\rho,T,\epsilon)\left(\sum_{i=k+1}^{\infty}h_{i}^2(1)\right)^{\frac{\rho}{2}},
  \end{aligned}
\end{equation}
Note that $\sum_{i=1}^{\infty}h_{i}(y)^{2}=y$, for $y\in [0,1]$, then we need to choose $\{k_{n^\prime}\}_{n^\prime\in\mathbb{Z}^{+}}$such that
$\sum_{i=k_{n^{\prime}}+1}^{\infty}h_{i}^{2}(1)<2^{-4n^{\prime}}$ holds.
By the Chebyshev's  inequality   we have that
\begin{equation}
  \begin{aligned}
    \mathbb{P}(\|u^\epsilon_{n}(t)-u_{n}^{k_{n^\prime},\epsilon}(t)\|_{\rho}\geq 2^{-n^\prime} )&\leq \frac{\mathbb{E}\|u_{n}^{\epsilon}(t)-u_{n}^{k_{n^\prime},\epsilon}(t)\|_{\rho}^{\rho}}{2^{-n^\prime\rho}}
    &\leq C(n,\rho,T,\epsilon)2^{-n^\prime\rho}.
  \end{aligned}
\end{equation}
By the Borel-Cantelli lemma, we conclude that there exists a set $\Omega^{(n)}\subset \Omega, n\in \mathbb{Z}^{+}$  of full
$\mathbb{P}$-measure  and $n^{\prime}_{0}=n_{0}^{\prime}(\omega,n)$ such that for every $\omega\in \Omega^{(n)}$ ,
$$\|u^{\epsilon}_n(t)-u_{n}^{k_{n^\prime},\epsilon}(t)\|_{\rho}\leq 2^{-n^{\prime}}$$
for $n^{\prime}\geq n^{\prime}_{0}$. Furthermore, taking $\hat{\Omega}^{0}=\bigcap_{n\geq 1}\Omega^{(n)}$, then $\mathbb{P}(\hat{\Omega}^{0})=1$.
\end{proof}

\begin{lemma}\label{Lemma 3.1}
For $n\geq 0,\rho\geq 2$, there is a constant $C>0$ such that
\begin{equation}\nonumber
\begin{aligned}
\mathbb{E}(\sup_{0\leq t\leq T}\|u^{\epsilon}_{n}(t)\|^{\rho}_{\rho})\leq C(1+\|\eta_{0,n}\|^{\rho}_{\rho}).
\end{aligned}
\end{equation}
In particular, the random variable $\sup_{0\leq t\leq T}\|u^{\epsilon}_{n}(t)\|_{\rho}^{\rho}$ is bounded in probability uniformly in $n$.
\end{lemma}
\begin{proof}
According to Lemma 4.3, we obtain
\begin{equation}\label{11}
  \begin{aligned}
    \|u_{n}^{k_{n^\prime},\epsilon}(t)\|_{\rho}^{\rho}\leq & \|\eta_{0,n}\|_{\rho}^{\rho}+\rho(\rho-1)\int_{0}^{t}\int_{0}^{1}|u_{n}^{k_{n^\prime},\epsilon}(s,x)|^{\rho-2}\left|\frac{\partial u_{n}^{k_{n^\prime},\epsilon}(s,x)}{\partial x}\right|^{2}dxds\\
    &+\rho\int_{0}^{t}\int_{0}^{1}|u_{n}^{k_{n^\prime},\epsilon}(s,x)|^{\rho-2}u_n^{k_{n^\prime},\epsilon}(s,x)f_{n}(s,x,u_n^{k_{n^\prime},\epsilon}(s,x))dxds\\
    &-\rho(\rho-1)\int_{0}^{t}\int_{0}^{1}|u_{n}^{k_{n^\prime},\epsilon}(s,x)|^{\rho-2}\frac{\partial}{\partial x}u_{n}^{k_{n^\prime},\epsilon}(s,x)g_{n}(s,x,u_{n}^{k_{n^\prime},\epsilon}(s,x))dxds\\
    &+\rho\sqrt{\epsilon}\sum_{i=1}^{k_{n^\prime}}\int_{0}^{t}\int_{0}^{1}|u_{n}^{k_{n^\prime},\epsilon}(s,x)|^{\rho-2}u_{n}^{k_{n^\prime},\epsilon}(s,x)\sigma_{n}(s,x,u_{n}^{k_{n^\prime},\epsilon}(s,x))h_{i}(x)dxdw_{i}(s)\\
    &+\frac{\sqrt{\epsilon}}{2}\rho(\rho-1)\int_{0}^{t}\int_{0}^{1}|u_{n}^{k_{n^\prime},\epsilon}(s,x)|^{\rho-2}\sigma^2_{n}(s,x,u_{n}^{k_{n^\prime},\epsilon}(s,x))dxds\\
    =&\|u_{0,n}\|_{\rho}^{\rho}+\tilde{A}_1(t)+\tilde{A}_2(t)+\tilde{A}_3(t)+\tilde{A}_4(t)+\tilde{A}_5(t)\quad  \mathbb{P}-a.s.
  \end{aligned}
\end{equation}
By the assumption \textbf{(H5)} and Young's inequality, we have
\begin{equation}
\begin{aligned}\label{3.17}
|\tilde{A}_{2}(t)|&\leq\rho\int_{0}^{t}\int_{0}^{1}|f(s,x,u_{n}^{k_{n^\prime},\epsilon}(s,x)||u_{n}^{k_{n^\prime},\epsilon}(s,x)|^{\rho-1}dxds\\
&\leq  C(\rho,K,T)+C(\rho,K,T)\int_{0}^{t}\|u_{n}^{k_{n^\prime},\epsilon}(s)\|_{\rho}^{\rho}ds.
\end{aligned}
\end{equation}
Similarly,
\begin{equation}\label{13}
  \begin{aligned}
    |\tilde{A}_5(t)|\leq C(\rho,K,T,\epsilon)+C(\rho,K,T,\epsilon)\int_{0}^{t}\|u_{n}^{k_{n^\prime},\epsilon}(s)\|_{\rho}^{\rho}ds.
  \end{aligned}
\end{equation}
By the assumption \textbf{(H4)}, we have
\begin{equation}\nonumber
\begin{aligned}
\tilde{A}_{3}(t):=\tilde{A}_{3,1}(t)+\tilde{A}_{3,2}(t),
\end{aligned}
\end{equation}
where
\begin{equation}
  \begin{aligned}
    \tilde{A}_{3,1}(t)=-\rho(\rho-1)\int_{0}^{t}\int_{0}^{1}|u_{n}^{k_{n^\prime},\epsilon}(s,x)|^{\rho-2}\frac{\partial}{\partial x}u_{n}^{k_{n^\prime},\epsilon}(s,x)g_{1,n}(s,x,u_n^{k_{n^\prime},\epsilon}(s,x))dxds
  \end{aligned}
\end{equation}
and
\begin{equation}
  \begin{aligned}
    \tilde{A}_{3,2}(t)=-\rho(\rho-1)\int_{0}^{t}\int_{0}^{1}|u_{n}^{k_{n^\prime},\epsilon}(s,x)|^{\rho-2}\frac{\partial}{\partial x}u_{n}^{k_{n^\prime},\epsilon}(s,x)g_{2,n}(s,u_n^{k_{n^\prime},\epsilon}(s,x))dxds.
  \end{aligned}
\end{equation}
In addition,
\begin{equation}
  \begin{aligned}
    \tilde{A}_{3,2}(t)=-\rho(\rho-1)\int_{0}^{t}\int_{0}^{1}\frac{d}{dx}\int_{0}^{u_{n}^{k_{n^\prime},\epsilon}(s,x)}|r|^{\rho-2}g_{2,n}(s,r)drdxds=0.
  \end{aligned}
\end{equation}
By the assumption \textbf{(H4)} and Young's inequality, we obtain
\begin{equation}
  \begin{aligned}
    |\tilde{A}_3(t)|&=|\tilde{A}_{3,1}(t)|\\
    &\leq \rho(\rho-1)\int_{0}^{t}\int_{0}^{1}|u_{n}^{k_{n^\prime},\epsilon}(s,x)|^{\rho-2}\left|\frac{\partial}{\partial x}u_{n}^{k_{n^\prime},\epsilon}(s,x)\right||g_{1,n}(s,x,u_n^{k_{n^\prime},\epsilon}(s,x))|dxds\\
    &\leq C(\rho,K)\int_{0}^{t}\int_{0}^{1}(1+|u_{n}^{k_{n^\prime},\epsilon}(s,x)|^{2})|u_{n}^{k_{n^\prime},\epsilon}(s,x)|^{\rho-2}dxds\\
    &~~~~+\rho(\rho-1)\int_{0}^{t}\int_{0}^{1}|\partial_x u_{n}^{k_{n^\prime},\epsilon}(s,x)|^{2}|u_{n}^{k_{n^\prime},\epsilon}(s,x)|^{\rho-2}dxds.
  \end{aligned}
\end{equation}
Thus,
\begin{equation}\label{18}
  \begin{aligned}
    \tilde{A}_{1}(t)+ |\tilde{A}_3(t)|\leq C(\rho,K,T)\left(1+\int_{0}^{t}\int_{0}^{1}|u_{n}^{k_{n^\prime},\epsilon}(s,x)|^{\rho}dxds\right).
  \end{aligned}
\end{equation}
Combining \eqref{11}-\eqref{13} and \eqref{18}, we obtain
\begin{equation}
  \begin{aligned}
    \mathbb{E}\|u_{n}^{k_{n^\prime},\epsilon}(t)\|_{\rho}^{\rho}\leq \|\eta_{0,n}\|_{\rho}^{\rho}+ C(\rho,K,T,\epsilon)+C(\rho,K,T,\epsilon)\int_{0}^{t}\mathbb{E}\|u_{n}^{k_{n^\prime},\epsilon}(s)\|_{\rho}^{\rho}ds.
  \end{aligned}
\end{equation}
According to Gr\"{o}nwall's inequality, it follows that
\begin{equation}\label{20}
  \begin{aligned}
    \sup_{t\in[0,T]}\mathbb{E}\|u_{n}^{k_{n^\prime},\epsilon}(t)\|_{\rho}^{\rho}\leq  C(\rho,K,T,\epsilon)(C(\rho,K,T,\epsilon)+\|\eta_{0,n}\|_{\rho}^{\rho}).
  \end{aligned}
\end{equation}
Note that our aim is to estimate $\mathbb{E}\left(\sup_{t\in[0,T]}\|u^{k_{n^\prime},\epsilon}_{n}(t)\|_{\rho}^{\rho}\right)$, based on   \eqref{20}, we  can get the estimate for $\mathbb{E}\left(\sup_{t\in[0,T]}\|u^{k_{n^\prime},\epsilon}_{n}(t)\|_{\rho}^{\rho}\right)$.
Similarly, we only deal with the stochastic term, and the other terms  can be  dealt as above.\\
Applying  to  B-D-G inequality,  Young's inequality, H\"{o}lder's continuity, Young's inequality,  and \eqref{20} to $\tilde{A}_{4}(t,x)$
\begin{equation}
  \begin{aligned}
    \mathbb{E}&\sup_{t\in[0,T]}\left|\rho\sqrt{\epsilon}\sum_{i=1}^{k_{n^\prime}}\int_{0}^{t}\int_{0}^{1}|u_{n}^{k_{n^\prime},\epsilon}(s,x)|^{\rho-2}u_{n}^{k_{n^\prime},\epsilon}(s,x)\sigma_{n}(s,x,u_{n}^{k_{n^\prime},\epsilon}(s,x))h_{i}(x)dxdw_{i}(s)\right|\\
    &\leq\rho\sqrt{\epsilon}\mathbb{E}\left(\sum_{i=1}^{k_{n^\prime}}\int_{0}^{T}\left(\int_{0}^{1}\sigma_{n}(s,x,u_{n}^{k_{n^\prime},\epsilon}(s,x))|u_{n}^{k_{n^\prime},\epsilon}(s,x)|^{\rho-1}h_{i}(x)dx\right)^{2}ds\right)^{\frac{1}{2}}\\
    &\leq C(\rho,K,\epsilon)\mathbb{E}\left( \sum_{i=1}^{k_{n^\prime}} \int_{0}^{T}\left(\int_{0}^{1}(|u_{n}^{k_{n^\prime},\epsilon}(s,x)|^{\rho}+1)h_{i}(x)dx\right)^{2}ds\right)^{\frac{1}{2}}\\
    &\leq  C(\rho,K,T,\epsilon)+ C(\rho,K,\epsilon)\mathbb{E}\left( \sum_{i=1}^{k_{n^\prime}} \int_{0}^{T}\int_{0}^{1}|u_{n}^{k_{n^\prime},\epsilon}(s,x)|^{\rho}h^2_i(x)dx\|u_{n}^{k_{n^\prime},\epsilon}(s)\|_{\rho}^{\rho}ds\right)^{\frac{1}{2}}\\
    &\leq  C(\rho,K,T,\epsilon)+\frac{1}{2}\mathbb{E}\sup_{t\in[0,T]}\|u_{n}^{k_{n^\prime},\epsilon}(s)\|_{\rho}^{\rho} +C(\rho,K,\epsilon)\int_{0}^{T}\mathbb{E}\|u_{n}^{k_{n^\prime},\epsilon}(s)\|_{\rho}^{\rho} ds\\
    &\leq \frac{1}{2}\mathbb{E}\sup_{t\in[0,T]}\|u_{n}^{k_{n^\prime},\epsilon}(s)\|_{\rho}^{\rho}  +C(\rho,K,T,\epsilon)(C(\rho,K,T,\epsilon)+\|\eta_{0,n}\|_{\rho}^{\rho}).
  \end{aligned}
\end{equation}
Hence, we have
$$\mathbb{E}\sup_{t\in[0,T]}\|u_{n}^{k_{n^\prime},\epsilon}(t)\|_{\rho}^{\rho} \leq C(\rho,K,T,\epsilon)(C(\rho,K,T,\epsilon)+\|\eta_{0,n}\|_{\rho}^{\rho}),$$
then  Lemma \ref{convergence}  shows that
\begin{equation}\label{23}
  \begin{aligned}
    \mathbb{E}\sup_{t\in[0,T]}\|u^\epsilon_{n}(t)\|_{\rho}^{\rho}&\leq C(\rho,K,T,\epsilon)(C(\rho,K,T,\epsilon)+\|\eta_{0,n}\|_{\rho}^{\rho})\\
    &\leq C(\rho,K,T,\epsilon)(1+\|\eta_{0,n}\|_{\rho}^{\rho}).
  \end{aligned}
\end{equation}
By the above  \eqref{23},  the random variable $\sup_{t\in[0,T]}\|u_{n}(t)\|_{\rho}^{\rho}$   is bounded  in probability, uniformly  in $n$.
\end{proof}
The next Lemma is used in proving  the tightness of $J_{2,n}^{\epsilon}$ and $J_{3,n}^{\epsilon}$ in equation \eqref{3.16*}, it can be found in \cite{MR1608641}.
\begin{lemma}\label{Tight tool}
Let $\zeta_{n}(t,x)$ be a sequence of random fields defined on $[0,T]$ such that $\sup_{t\in[0,T]}\|\zeta_{n}(t,\cdot)\|_{q} \\\leq\theta_{n}$ for any $\rho\in[1,\infty)$ and $q\in[1,\rho)$, where $\theta_{n}$ is a finite random variable. If $\theta_{n}$ is uniformly bounded in probability. Then
the sequence $J(\zeta_{n}):=\int_{0}^{t}\int_{0}^{1}R(r,t,x,y)\zeta_{n}(r,y)dydr$, where $R(r,t,x,y)=G(r,t,x,y)$ or  $R(r,t,x,y)=\partial_{y}G(r,t,x,y)$, is uniformly tight in $C([0,T];L^{\rho}([0,1]))$.
\end{lemma}
For $J_{1,n}$, we know that it is tight by convolution operator $G_{t}$ is uniformly  bounded. Finally, we need to illustrate that $J_{4,n}^{\epsilon}$ has also such property.
\begin{lemma}\label{tightness of stochatic term}
For the previous initial data $\eta_{0,n}\in L^{\rho}([0,1]),\rho> 6$, the stochastic convolution  term $J^{\epsilon}_{4,n}$ is uniformly tight in $C([0,T]\times[0,1])$.
\end{lemma}
\begin{proof}
The proof of the Lemma is an application of Arzela-Ascoli theorem or Aldous' tightness criterion.  We first show the uniform boundedness of the stochastic terms $J^{\epsilon}_{4,n}$.
For any $x\in[0,1], t\in[0,T]$, using Burkholder-Davis-Gundy inequality, H\"{o}lder's inequality and  H\"{o}lder's inequality  with exponent $p$ and $q\leq \frac{\rho}{2}$, Lemma \ref{Lemma 3.1}, we have
\begin{equation}
\begin{aligned}
&\mathbb{E}\left|\int_{0}^{t}\int_{0}^{1}G_{t-s}(x,y)\sigma(s,y,u_{n}^{\epsilon}(s,y))W(dy,ds)\right|^{\rho}\leq\mathbb{E}\left|\int_{0}^{t}\int_{0}^{1}G^{2}_{t-s}(x,y)\sigma^{2}(s,y,u_{n}^{\epsilon}(s,y))ydyds\right|^{\frac{\rho}{2}}\\
&~~\leq \left(\int_{0}^{t}\int_{0}^{1}G_{t-s}^{2p}(x,y)dyds\right)^{\frac{\rho}{2p}}\mathbb{E}\left(\int_{0}^{t}\int_{0}^{1}\sigma^{2q}(s,y,u_{n}^{\epsilon}(s,y))dyds\right)^{\frac{\rho}{2q}}\\
&~~\leq t^{\frac{(3-2p)\rho}{4p}}t^{\frac{\rho}{2q}-1}\mathbb{E}\int_{0}^{t}\int_{0}^{1}\sigma^{\rho}(s,y,u^{\epsilon}_{n}(s,y))dyds\\
&~~\leq C(\rho,T)(1+\mathbb{E}(\sup_{t\in[0,T]}\|u_{n}^{\epsilon}(t)\|_{\rho}^{\rho}))\leq C(\rho,T,\eta_{0,n}).
\end{aligned}
\end{equation}
And then, we need to illustrate the continuity of the variable $t$ and $x$. For any $0\leq s<t\leq T$, we have
\begin{equation}
\begin{aligned}
&\int_{0}^{t}\int_{0}^{1}G_{t-s^{\prime}}(x,y)\sigma(s^{\prime},y,u_{n}^{\epsilon}(s^{\prime},y))W(ds^{\prime},dy)-\int_{0}^{s}\int_{0}^{1}G_{s-s^{\prime}}(x,y)\sigma(s^{\prime},y,u_{n}^{\epsilon}(s^{\prime},y))W(ds^{\prime},dy)\\
&~~=\int_{s}^{t}\int_{0}^{1}G_{t-s^{\prime}}(x,y)\sigma(s^{\prime},y,u_{n}^{\epsilon}(s^{\prime},y))W(ds^{\prime},dy)\\
&~~~~+\int_{0}^{s}\int_{0}^{1}(G_{t-s^{\prime}}(x,y)-G_{s-s^{\prime}}(x,y))\sigma(s^{\prime},y,u_{n}^{\epsilon}(s^{\prime},y))W(ds^{\prime},dy)\\
&~~=: I_{1}+I_{2}.
\end{aligned}
\end{equation}
Using Burkholder-Davis-Gundy inequality, H\"{o}lder's inequality, Lemma \ref{lemma 2.1} and Lemma \ref{Lemma 3.1}, we obtain
\begin{equation}
\begin{aligned}
\mathbb{E}|I_{2}|^{\rho}&\leq C(T,\rho)|t-s|^{\frac{(3-2p)\rho}{2p}}\mathbb{E}\left(\int_{0}^{s}\int_{0}^{1}\sigma^{2q}(s^{\prime},y,u^{\epsilon}_{n}(s^{\prime},y))dyds\right)^{\frac{\rho}{2q}}\\
&\leq C(T,\rho)|t-s|^{\frac{(3-2p)\rho}{4p}}\mathbb{E}\int_{0}^{s}\int_{0}^{1}\sigma^{\rho}(s^{\prime},y,u^{\epsilon}_{n}(s^{\prime},y))ds^{\prime}dy\\
&\leq C(T,\rho)|t-s|^{\frac{(3-2p)\rho}{4p}}(1+\mathbb{E}\sup_{0\leq s\leq T}\|u_{n}^{\epsilon}(s)\|_{\rho}^{\rho})\leq C(T,\rho,\eta_{0,n})|t-s|^{\frac{(3-2p)\rho}{4p}},
\end{aligned}
\end{equation}
where $\frac{1}{p}+\frac{1}{q}=1$,$\frac{\rho}{2q}>1$ and $1<2p<3$, thus $\rho>6$. For term $I_{2}$, using the same estimate, we have
\begin{equation}
\begin{aligned}
\mathbb{E}|I_{1}|^{\rho}&\leq C(T,\rho)\left(\int_{s}^{t}\int_{0}^{1}G_{t-s}^{2p}(x,y)dyds\right)^{\frac{\rho}{2p}}\mathbb{E}\left(\int_{s}^{t}\int_{0}^{1}\sigma^{2q}(s^{\prime},y,u^{\epsilon}_{n}(s^{\prime,y})) dyds\right)^{\frac{\rho}{2q}}\\
&\leq C(T,\rho)|t-s|^{\frac{(3-2p)\rho}{4p}}(1+\mathbb{E}\sup_{0\leq s\leq T}\|u_{n}^{\epsilon}(s)\|_{\rho}^{\rho})\leq C(T,\rho,\eta_{0,n})|t-s|^{\frac{(3-2p)\rho}{4p}},
\end{aligned}
\end{equation}
where $C(T,\rho,\eta_{0,n})$ is uniformly bounded with respect to $n$, since $\eta_{0,n}$ converge to $\eta$.  For any $0\leq x_{1}<x_{2}\leq 1$, similar to the above estimate,  we obtain the continuity of the space variable,
\begin{equation}
\begin{aligned}
&\mathbb{E}\left|\int_{0}^{t}\int_{0}^{1}\left[G_{t-s}(x_{1},y)-G_{t-s}(x_{2},y)\right]\sigma(s,y,u_{n}^{\epsilon}(s,y))W(dy,ds)\right|^{\rho}\\
&\leq C(\rho,T)|x_{1}-x_{2}|^{\frac{(3-2p)\rho}{2p}}(1+\mathbb{E}\sup_{0\leq s\leq T}\|u_{n}^{\epsilon}(s)\|_{\rho}^{\rho})\leq C(T,\rho,\eta_{0,n})|t-s|^{\frac{(3-2p)\rho}{2p}}.
\end{aligned}
\end{equation}
Hence, Chebyshev's inequality  and  uniform boundedness of the stochastic terms  shows that
\begin{align*}
\lim_{K\rightarrow \infty}\limsup_{n}\mathbb{P}(\sup_{t\in[0,T],x\in[0,1]}|J_{4,n}^\epsilon|\geq K)=0
\end{align*}
Furthermore, under parabolic metric $|(t,x)|=|t|^{\frac{1}{2}}+|x|$, the above continuity tells us that there exists a $\delta(\epsilon)=\varepsilon^{\frac{4p}{3-2p}}C^{-\frac{2p}{\rho(3-2p)}}(T,\rho,\eta_{0,n})$ for any $\varepsilon>0$ such that $\delta<\delta(\epsilon)$
\begin{align*}
\limsup_{n}\mathbb{P}(\sup_{t\in[0,T],x\in[0,1]\atop
 |(t,x)|\leq \delta}|J_{4,n}^\epsilon(t,x)-J_{4,n}^\epsilon(s,y)|\geq \varepsilon)\leq \varepsilon^{\rho},
\end{align*}
namely,
\begin{align*}
\lim_{\delta\rightarrow 0}\limsup_{n}\mathbb{P}(\sup_{t\in[0,T],x\in[0,1]\atop
 |(t,x)|\leq \delta}|J_{4,n}^\epsilon(t,x)-J_{4,n}^\epsilon(s,y)|\geq \varepsilon)=0.
\end{align*}
Thus, $J_{4,n}^\epsilon$ is tight in $C([0,T]\times[0,1])$.
\end{proof}
From the above arguments we know that $u_{n}^{\epsilon}$ is also uniformly tight in $C([0,T];L^{\rho}(0,1))$ for all $n$. Based on the previous results we are able to show the well-posedness  of the equation \eqref{1.1}. Hence we have the following result.
\begin{theorem}\label{existence and unique}
Let $f,g,\sigma$ satisfy the assumption \textbf{H}, For any $\rho>6$ and initial data $\eta(x)\in L^{\rho}([0,1])$, equation \eqref{1.1} has a unique $\mathcal{F}_t$-adapted $L^{\rho}([0,1])$-valued continuous solution.
\end{theorem}
\begin{proof}
The uniqueness is very trivial, we can get this fact from Lemma \ref{Lemma 3.1}. Indeed, let $u$ and $v$ are solution of equation $\eqref{1.1}$ and $\tau_{R}^{1}:=\inf\{t\geq 0:\|u(t)\|_{\rho}\geq R\}\wedge T$, $\tau_{R}^{2}:=\inf\{t\geq 0:\|v(t)\|_{\rho}\geq R\}\wedge T$. With the help of Theorem \ref{Theorem 3.1}, we have $u=v$ for any $t\in [0,\tau_{R}^{1}\wedge\tau_{R}^{2}]$. Finally, by Lemma \ref{Lemma 3.1} and Chebyshev's inequality, we have  $P(\tau_{R}^{1}\wedge\tau_{R}^{2}<T)\rightarrow 0$ as $R\rightarrow\infty$. Thus we obtain the uniqueness of the solution on the interval $[0,T]$. \\
Finally, we consider the existence of the solution $u^{\epsilon}$. The main tool is the Lemma 1.1\cite{MR1392450} and Skorohod representation theorem.  In view of Lemma \ref{Tight tool}, \ref{tightness of stochatic term} and $J_{1,n}$ is uniformly tight for $n$,  then $u^\epsilon_n$ is tight. Thus for any two subsequence $u_l^\epsilon$ and $u_m^\epsilon$, then Prokhorov's theorem and Skorohod representation theorem show that there exists subsequence $l(k),m(k)$ of $l,m$, and a sequence of random  elements $z_k:=(\tilde{u}_{k}^\epsilon,\bar{u}_{k}^{\epsilon},\frac{\partial\hat{W}_k}{\partial x})_{k\geq 1}\in\mathbb{E}^{\prime }\times\mathbb{E}^{\prime}\times C([0,T]\times[0,1])$, where $\mathbb{E}^{\prime}:=C([0,T];L^{\rho}([0,1]))$. And $z_k$ converges to  a random element $z:=(\tilde{u}^\epsilon,\bar{u}^{\epsilon},\frac{\partial \hat{W}}{\partial x})$ with probability one on some probability space $(\hat{\Omega},\hat{\mathcal{F}},\hat{\mathbb{P}})$.  Furthermore, the law of $\left(u_{l(k)},u_{m(k)},\frac{\partial W}{\partial x}\right)$ coincide with $z_{k}$. Note that the filtration $\hat{\mathcal{F}}_{t}$ and $\hat{\mathcal{F}}$  are generated by $z(s,x),x\in[0,1]$ and $z_k(s,x),x\in[0,1]$, respectively.  Based on the the Lemma 1.1\cite{MR1392450}, it  needs only to check $z_{k}$ weakly converge. Due to the weak solution is equivalent to the mild solution, its proof can be found in \cite{MR1608641}.  Then, taking the limit $k\rightarrow\infty$ for the weak formulation, we obtain the weak formulation
\begin{align*}
  \int_{0}^{1}\tilde{u}^{\epsilon}(t,x)\varphi(x)dx&=\int_{0}^{1}u_{0}(x)\varphi(x)dx+\int_{0}^{t}\int_{0}^{1}\tilde{u}^{\epsilon}(s,x)\partial_x^2\varphi(x)dxds\\
  &\quad+\int_{0}^{t}\int_{0}^{1}f(s,x,\tilde{u}^{\epsilon}(s,x))\varphi(x)dxds\\
   &\quad-\int_{0}^{t}\int_{0}^{1}g(s,x,\tilde{u}^{\epsilon}(s,x))\partial_x\varphi(x)dxds\\
   &\quad+\int_{0}^{t}\int_{0}^{1}\sigma(s,x,\tilde{u}^{\epsilon}(s,x))\varphi(x)\hat{W}(dx,ds) \quad a.s.
\end{align*}
for all $t\in[0,T]$ on $\left(\hat{\Omega},\hat{\mathcal{F}},\hat{\mathcal{F}}_{t},\hat{\mathbb{P}}\right)$. In addition, $\bar{u}^\epsilon$ also satisfy the above equation, by the uniqueness of the solution, i.e. $\bar{u}^\epsilon=\tilde{u}^\epsilon$  and the main technique we mentioned, then $u_{n}^\epsilon$ converges  to some random element $u \in\mathbb{E}^{\prime}$ in  probability.  Hence, we complete the proof of the existence.
\end{proof}
\section{A criteria for large deviation Principle}
First we give a general framework of the large deviation principle.
\begin{definiton}[Large Deviation Principle]
Let $I:\mathcal{X}\rightarrow [0,\infty]$ be a rate function on Polish space $\mathcal{X}$. The sequence of random variable $\{X^{\epsilon}\}$ satisfies the large principle on $\mathcal{X}$ with rate function $I$ if
\begin{description}
  \item[1] For any closed subset $F\subset \mathcal{X}$,
  \begin{equation}\nonumber
  \limsup_{\epsilon\rightarrow 0}\log \mathbb{P}(X^{\epsilon}\in F)\leq-\inf_{x\in F}I(x).
  \end{equation}
  \item[2] For any open set $G\subset\mathcal{X}$,
  \begin{equation}\nonumber
  \liminf_{\epsilon\rightarrow 0}\log \mathbb{P}(X^{\epsilon}\in G)\geq-\inf_{x\in G}I(x).
  \end{equation}
\end{description}
\end{definiton}
The proof of the large deviation principle of equation \eqref{1.1} is based on the weak convergence method. For this approach, one proves the Laplace principle which is equivalent to the large deviation principle.
\begin{definiton}[Laplace principle]
The sequence of random $\{X^{\epsilon}\}$ on polish space $\mathcal{X}$ is said to satisfy the Laplace principle  with rate function $I$ if for each $h\in C_{b}(\mathcal{X};\mathbb{R})$,
\begin{equation}\nonumber
\lim_{\epsilon\rightarrow 0}\epsilon\log\mathbb{E}\left\{\exp\left[-\frac{h(X^{\epsilon})}{\epsilon}\right]\right\}=-\inf_{x\in\mathcal{X}}\left\{h(x)+I(x)\right\}.
\end{equation}
\end{definiton}
Let $(\Omega,\mathcal{F},\mathbb{P},\mathcal{F}_{t})$ be a filtered probability space which emerges in  previous section, a predictable process $\psi:\Omega\times [0,T]\rightarrow L^{2}([0,T]\times[0,1];\mathbb{R}) $. Introducing the following sets
\begin{equation}\nonumber
\begin{aligned}
&S^{N}:=\left\{\psi\in L^{2}([0,T]\times[0,1];\mathbb{R}):\int_{0}^T\int_{0}^{1}|\psi(s,y)|^{2}dyds\leq N\right\},\quad N\in\mathbb{N},\\
&\mathcal{A}_{2}:=\left\{\psi:\int_{0}^T\int_{0}^{1}|\psi(s,y)|^{2} dyds<\infty\quad\mathbb{P}- \text{a.s.}\right\},\\
&\mathcal{A}_{2}^{N}:=\left\{\psi\in\mathcal{A}_{2}:\int_{0}^{T }\int_{0}^{1}|\psi(s,y)|^{2} dyds\leq N\quad \mathbb{P}-\text{a.s.}\right\},
\end{aligned}
\end{equation}
where $S^{N}$ is a compact metric space, endowed with the weak topology on $L^{2}([0,T]\times[0,1];\mathbb{R})$, $\mathcal{A}_{2}^{N}$ is an admissible control set. Let $\mathcal{X}$ and $\mathcal{X}_{0}$ be Polish spaces. Assume that initial data takes value in a compact subspace of $\mathcal{X}_{0}$ and the solution in space $\mathcal{X}$. For each $\epsilon>0$, let $\mathcal{G}^{\epsilon}:\mathcal{X}_{0}\times C([0,T]\times[0,1])\rightarrow \mathcal{X}$ be a family of measurable maps. Define $Z^{\epsilon,\eta}=\mathcal{G}^{\epsilon}(\eta,\sqrt{\epsilon}\frac{\partial W}{\partial x})$. We state below a sufficient condition for the uniform Laplace principle for the  family $Z^{\epsilon,\eta}$.\\
\textbf{Condition.} There exists a measurable map $\mathcal{G}^{0}:\mathcal{X}_{0}\times C([0,T]\times[0,1])\rightarrow \mathcal{X}$ such that the following conditions hold
\begin{description}
  \item[(a)] For $M\in \mathbb{N}$, let $\psi^{n},\psi\in S^{M}$, and $(\eta^{n},\psi^{n})\rightarrow (\eta,\psi)$ in distribution  as $n\rightarrow\infty$. Then
  $$\mathcal{G}^{0}\left(\eta^{n},\int_{0}^{t}\int_{0}^{x}\psi^{n}(s,y) ds\right)\rightarrow \mathcal{G}^{0}\left(\eta,\int_{0}^{t}\int_{0}^{x}\psi(s,y)dyds\right),$$
  in distribution as $n\rightarrow\infty$.
  \item[(b)] For $M\in \mathbb{N}$, let $\psi^{\epsilon},\psi\in\mathcal{A}_{2}^{N}$ be such that $\psi^{\epsilon}\rightarrow \psi$ in distribution, and $\eta^{\epsilon}\rightarrow\eta$,  as $\epsilon \rightarrow0$. Then
   $$\mathcal{G}^{\epsilon}\left(\eta^{\epsilon},\sqrt{\epsilon}\frac{\partial W}{\partial x}+
   \int_{0}^{t}\int_{0}^{x}\psi^{\epsilon}(s,y)ds\right)\rightarrow \mathcal{G}^{0}\left(\eta,\int_{0}^{t}\int_{0}^{x}\psi(s,y)ds\right),$$
  in distribution as $\epsilon\rightarrow 0$.
\end{description}
For $\phi\in \mathcal{X}$ and $\eta\in\mathcal{X}_{0}$, define $\mathbb{S}_{\phi}:=\left\{\psi\in L^{2}([0,T]\times[0,1];\mathbb{R}):\phi=\mathcal{G}^{0}(\eta,\int_{0}^{t} \int_{0}^{x}\psi(s,y)dyds)\right\}$. Let $I_{\eta}:\mathcal{X}\rightarrow [0,\infty]$ be defined by
\begin{equation}\label{rate function}
I_{\eta}(\phi)=\left\{\begin{array}{ll}
\inf_{\phi\in \mathbb{S}_{\phi}} \left\{\frac{1}{2}\int_{0}^{T}\int_{0}^{1}|\psi(s,y)|^{2}dyds \right\} & \text { if } \phi \in \mathcal{X}, \\
+\infty & \text { if } \mathbb{S}_{\phi}=\emptyset.
\end{array}\right.
\end{equation}

The following criteria was established in \cite[Theorem 7]{MR2435853}.
\begin{lemma}\label{Argument tool}
Let $\mathcal{G}^{0}:\mathcal{X}_{0}\times C([0,T]\times[0,1])\rightarrow \mathcal{X}$ be a measurable maps, and assume that the above conditions hold. Then the family $\{Z^{\epsilon,\eta}\}$ satisfies the uniform Laplace principle on $\mathcal{X}$ with rate function \eqref{rate function}, uniformly $\eta$ on compact subsets of $\mathcal{X}_{0}$.
\end{lemma}
If we can verify condition $(a)$ and $(b)$, then a family of solution $\{Z^{\epsilon,\eta}\}$ satisfies the uniform Laplace principle,  then we obtain the large deviation principle for $\{Z^{\epsilon,\eta}\}$ by \cite{MR2435853}. Hence our next aim to verify condition $(a)$ and $(b)$. Before that, we introduce some results for the controlled and skeleton equations.
\section{Controlled and skeleton equations}
In this section, we mainly introduce some results on the controlled and skeleton equations. Let $\mathcal{X}=C([0,T],L^{\rho}([0,1]))$ and $\mathcal{X}_{0}=L^{\rho}([0,1])$. And the solution map of \eqref{1.1} $u^{\epsilon}=\mathcal{G}^{\epsilon}(\eta,\sqrt{\epsilon}\frac{\partial W}{\partial x})$. Denote $v^{\epsilon,\psi}_{\eta}(t,x):=\mathcal{G}^{\epsilon}\left(\eta,\int_{0}^{t}\int_{0}^{x}\psi(s,y)ds+\sqrt{\epsilon}\frac{\partial W}{\partial x}\right)$, and it represents the solution of the controlled equation and it has the mild solution as follow
\begin{equation}\label{controlled equation}
\begin{aligned}
v^{\epsilon,\psi}_{\eta}(t,x)=&\int_{0}^{1}G_{t}(x,y)\eta(y)dy-\int_{0}^{t}\int_{0}^{1}\partial_{y}G_{t-s}(x,y)g(s,y,v^{\epsilon,\psi}_{\eta}(s,y))dyds\\
&+\int_{0}^{t}\int_{0}^{1}G_{t-s}(x,y)f(s,y,v^{\epsilon,\psi}_{\eta}(s,y))dyds\\
&+\int_{0}^{t}\int_{0}^{1}G_{t-s}(x,y)\sigma(s,y,v^{\epsilon,\psi}_{\eta}(s,y))\int_{0}^{y}\psi(s,y^\prime)dy^\prime dyds\\
&+\sqrt{\epsilon}\int_{0}^{t}\int_{0}^{1}G_{t-s}(x,y)\sigma(s,y,v^{\epsilon,\psi}_{\eta}(s,y))W(dy,ds).
\end{aligned}
\end{equation}
In particular, for $\epsilon=0$, we obtain the skeleton equation or limiting equation,  denote $v_{\eta}^{0,\psi}:=\mathcal{G}^{0}\left(\eta,\int_{0}^{t}\int_{0}^{x}\psi(s,y)dyds\right)$, with the following mild formulation
\begin{equation}\label{skeleton equation}
\begin{aligned}
v^{0,\psi}_{\eta}(t,x)=&\int_{0}^{1}G_{t}(x,y)\eta(y)dy-\int_{0}^{t}\int_{0}^{1}\partial_{y}G_{t-s}(x,y)g(s,y,v^{0,\psi}_{\eta}(s,y))dyds\\
&+\int_{0}^{t}\int_{0}^{1}G_{t-s}(x,y)f(s,y,v^{0,\psi}_{\eta}(s,y))dyds\\
&+\int_{0}^{t}\int_{0}^{1}G_{t-s}(x,y)\sigma(s,y,v^{0,\psi}_{\eta}(s,y))\int_{0}^{y}\psi(s,y^\prime)dy^\prime dyds.
\end{aligned}
\end{equation}
In addition, we need to specify the rate function. Let $\phi(t,x)\in C([0,T];L^{\rho}([0,1]))$. We define the following rate function
\begin{equation}\label{rate}
I_{\eta}(\phi):=\frac{1}{2}\inf_{\psi}\int_{0}^{T}\int_{0}^{1}|\psi(s,y)|^{2}dyds,
\end{equation}
where $\psi\in L^{2}([0,T]\times[0,1];\mathbb{R})$ such that $\phi$ satisfies the following relation
\begin{equation}\label{ske}
\begin{aligned}
\phi(t,x)=&\int_{0}^{1}G_{t}(x,y)\eta(y)dy-\int_{0}^{t}\int_{0}^{1}\partial_{y}G_{t-s}(x,y)g(s,y,\phi(s,y))dyds\\
&+\int_{0}^{t}\int_{0}^{1}G_{t-s}(x,y)f(s,y,\phi(s,y)))dyds\\
&+\int_{0}^{t}\int_{0}^{1}G_{t-s}(x,y)\sigma(s,y,\phi(s,y))\int_{0}^{y}\psi(s,y^\prime)dy^\prime dyds.
\end{aligned}
\end{equation}
The following theorems states the existence and uniqueness of the controlled and skeleton equations.
\begin{theorem}
Let $\eta,f,g,\sigma$ satisfy condition \textbf{(H1)-(H6)} and $\psi\in\mathcal{A}_{2}^{N}$ , then the skeleton equation \eqref{controlled equation} has a unique solution on $C([0,T];L^{\rho}([0,1]))$ for $\rho>6$ and $\epsilon>0$.
\end{theorem}
\begin{proof}
Similar to \cite{MR2435853} and \cite{MR3575422}, we mainly use  Girsanov's theorem  to prove the well-posed of  controlled  equation. For any fixed $\psi\in\mathcal{A}_{2}^{N}$,  define

$$\frac{d\hat{\mathbb{P}}^{\psi,\epsilon}}{d\mathbb{P}}:=e^{\biggr\{-\frac{1}{\epsilon}\int_{0}^{T}\int_{0}^{1}\psi(s,y)\frac{\partial }{\partial y} W(ds,dy)-\frac{1}{2\epsilon}\int_{0}^{T}\int_{0}^{1}\psi^{2}(s,y)dyds\biggr\}},$$
where $\frac{\partial W(s,y)}{\partial y}$ is a Brownian sheet and  the integral $\int_{0}^{T}\int_{0}^{1}\psi(s,y)\frac{\partial }{\partial y} W(ds,dy)$ should be understood  in the sense of Walsh\cite{MR876085},  the stochastic process
$$\hat{W}(t,x):=\sqrt{\epsilon}\frac{\partial W(t,x)}{\partial x}+\int_{0}^{t}\int_{0}^{x}\psi(s,x)dxds,\quad (t,x)\in[0,T]\times[0,1]$$
is a Brownian sheet.  Indeed,
since
$$e^{\biggr\{-\frac{1}{\epsilon}\int_{0}^{T}\int_{0}^{1}\psi(s,y)\frac{\partial }{\partial y} W(ds,dy)-\frac{1}{2\epsilon}\int_{0}^{T}\int_{0}^{1}\psi^{2}(s,y)dyds\biggr\}}$$
is a exponential martingale and  $\hat{\mathbb{P}}^{\psi,\epsilon}$ is a probability measure on  $(\Omega,\mathcal{F},\mathcal{F}_{t},\mathbb{P})$ and  $\hat{\mathbb{P}}^{\psi,\epsilon}$ is equivalent to $\mathbb{P}$. Thus Girsanov's theorem  shows that $\hat{W}(t,x)$ is a real-valued Wiener process with respect $\hat{\mathbb{P}}^{\psi,\epsilon}$,  by the equivalence of the measure, we  know that
$\hat{W}(t,x)$ is a Brownian sheet.
On the probability space $(\Omega,\mathcal{F},\mathcal{F}_{t},\hat{\mathbb{P}}^{\psi,\epsilon})$,  Theorem \ref{existence and unique} can be applied here,  we obtain the existence and uniqueness of the controlled equation for probability measure   $\hat{\mathbb{P}}^{\psi,\epsilon}$,  since
$\hat{\mathbb{P}}^{\psi,\epsilon}$ is equivalent to $\mathbb{P}$, then we get the well-posedness  of the skeleton equation  with respect to $\mathbb{P}$.
\end{proof}
\begin{theorem}
Let $\eta,f,g,\sigma$ satisfy condition \textbf{(H1)-(H6)} and $\psi\in L^{2}([0,T]\times[0,1];\mathbb{R})$, equation \eqref{skeleton equation} has a unique solution  on $C([0,T];L^{\rho}([0,1]))$ for $\rho>6$.
\end{theorem}
\begin{proof}
Due to $\psi(t,x)\in L^{2}([0,T]\times[0,1])$, then we get the well-posedness of the skeleton equation by the Theorem 4.2\cite{MR4246020}.
\end{proof}
\section{The proof of the Large deviation principle}
In this Section, we state the main result of this paper.  Based on the results of the previous section, we give our the main result as follow.
\begin{theorem}
Under condition \textbf{(H1)-(H6)} and $\rho>6$, the solution processes $\{u^{\epsilon}(t,x):t\in[0,T],x\in[0,1]\}$ satisfy the uniform Laplace principle on $C([0,T];L^{\rho}([0,1]))$  with rate function $I_{\eta}$ given by \eqref{rate}.
\end{theorem}
According to the Lemma \ref{Argument tool}, we need to the condition $(a)$ and $(b)$. Compared with condition $(a)$, the verification of condition $(b)$ is relatively complicated. We mainly give the verification of condition $(b)$.
\begin{proof}
The first thing we want to obtain is that the convergence of the controlled process, namely, let $M<\infty$ , suppose that $\eta^{\epsilon}\rightarrow\eta$ and $\psi^{\epsilon}\rightarrow \psi$ in distribution as $\epsilon\rightarrow 0$, where $\psi^{\epsilon}, \psi\in \mathcal{A}_{2}^{M}$. Is this $v^{\epsilon,\psi^{\epsilon}}_{\eta^{\epsilon}}\rightarrow v^{0,\psi}_{\eta}$ in distribution?  In order to answer the question, we adopt the compactness arguments.

Let
\begin{equation}
\begin{aligned}
v^{\epsilon,\psi^{\epsilon}}_{\eta^{\epsilon}}(t,x)=&\int_{0}^{1}G_{t}(x,y)\eta^{\epsilon}(y)dy-\int_{0}^{t}\int_{0}^{1}\partial_{y}G_{t-s}(x,y)g(s,y,v^{\epsilon,\psi^{\epsilon}}_{\eta^{\epsilon}}(s,y))dyds\\
&+\int_{0}^{t}\int_{0}^{1}G_{t-s}(x,y)f(s,y,v^{\epsilon,\psi^{\epsilon}}_{\eta^{\epsilon}}(s,y))dyds\\
&+\int_{0}^{t}\int_{0}^{1}G_{t-s}(x,y)\sigma(s,y,v^{\epsilon,\psi^{\epsilon}}_{\eta^{\epsilon}}(s,y))\int_{0}^{y}\psi^{\epsilon}(s,y^\prime)dy^\prime dyds\\
&+\sqrt{\epsilon}\int_{0}^{t}\int_{0}^{1}G_{t-s}(x,y)\sigma(s,y,v^{\epsilon,\psi^{\epsilon}}_{\eta^{\epsilon}}(s,y))W(dy,ds)\\
:=&J_{1}^{\epsilon}+J_{2}^{\epsilon}+J_{3}^{\epsilon}+J_{4}^{\epsilon}+J_{5}^{\epsilon}.
\end{aligned}
\end{equation}
For $J_{1}^{\epsilon}$, it is clear that $\|J_{1}^{\epsilon}\|_{\rho}\leq C\|\eta^{\epsilon}\|_{\rho}$. Then Markov's inequality shows that $\mathbb{P}(\|J_{1}^{\epsilon}\|_{\rho}>R)\leq \frac{\mathbb{E}\{\|J_{1}^{\epsilon}\|_{\rho}\}}{R}\leq \frac{C\|\eta^{\epsilon}\|_{\rho}}{R}$ for any $R>0$, this inequality implies $J_{1}^{\epsilon}$ is tight in $C([0,T];L^{\rho}([0,1]))$. For simplicity, we first deal with $J^{\epsilon}_{3}$. By Minkovski's inequality, Young's inequality and assumption $\mathbf{(H5)}$, one obtains
\begin{equation}
\begin{aligned}
\left\|\int_{0}^{t}\int_{0}^{1}G_{t-s}(\cdot,y)f(s,y,v^{\epsilon,\psi^{\epsilon}}_{\eta^{\epsilon}}(s,y))dyds\right\|_{\rho}&\leq \int_{0}^{t}\|G_{t-s}(\cdot)\|_{\rho}\|f(s,\cdot,v^{\epsilon,\psi^{\epsilon}}_{\eta^{\epsilon}}(s,\cdot))\|_{1}ds\\
&\leq C\int_{0}^{t}(t-s)^{-\frac{1}{2}+\frac{1}{2\rho}}ds\sup_{s\in[0,T]}\|f(s,\cdot,v^{\epsilon,\psi^{\epsilon}}_{\eta^{\epsilon}}(s,\cdot))\|_{1}\\
&\leq CT^{\frac{1}{2}+\frac{1}{2\rho}}(K+K\sup_{s\in[0,T]}\|v^{\epsilon,\psi^{\epsilon}}_{\eta^{\epsilon}}(s,\cdot))\|_{\rho}).
\end{aligned}
\end{equation}
In terms of Lemma \ref{Tight tool}, let $\zeta^{\epsilon}=K+K\sup_{s\in[0,T]}\|v^{\epsilon,\psi}_{\eta^{\epsilon}}(s,\cdot))\|_{\rho}$, then we  need only to illustrate the boundedness of the $\zeta^{\epsilon}$ in probability. Similar to the proof of the Lemma \ref{Lemma 3.1} and Theorem \ref{existence and unique}, consider the approximate process $v^{\epsilon,\psi^{\epsilon}}_{\eta^{\epsilon},n}(t,x)$, we have  $\sup_{t\in[0,T]}\|v^{\epsilon,\psi^{\epsilon}}_{\eta^{\epsilon},n}(t,\cdot)\|_{\rho}$  is bounded in probability and $v^{\epsilon,\psi^{\epsilon}}_{\eta^{\epsilon},n}(t,x)$ converge to $v^{\epsilon,\psi^{\epsilon}}_{\eta^{\epsilon}}(t,x)$ on $C([0,T];L^{\rho}([0,1]))$ in probability. Then we have
\begin{equation}
\begin{aligned}
\lim_{C\rightarrow\infty}\sup_{\epsilon\in[0,1]}\mathbb{P}(\zeta^{\epsilon}\geq C)&\leq \lim_{C\rightarrow\infty}\sup_{\epsilon\in[0,1]}\mathbb{P}\left(K\geq \frac{C}{2}\right)+\lim_{C\rightarrow\infty}\sup_{\epsilon\in[0,1]}\mathbb{P}\left(K\sup_{s\in[0,T]}\|v^{\epsilon,\psi^{\epsilon}}_{\eta^{\epsilon}}(s,\cdot))\|_{\rho}\geq\frac{C}{2}\right)\\
&=\lim_{C\rightarrow\infty}\sup_{\epsilon\in[0,1]}\mathbb{P}\left(K\sup_{s\in[0,T]}\|v^{\epsilon,\psi^{\epsilon}}_{\eta^{\epsilon}}(s,\cdot))\|_{\rho}\geq\frac{C}{2}\right)\\
&\leq \lim_{C\rightarrow\infty}\sup_{\epsilon\in[0,1]}\mathbb{P}\left(K\sup_{s\in[0,T]}(\|v^{\epsilon,\psi^{\epsilon}}_{\eta^{\epsilon},n}(s,\cdot))\|_{\rho}+\|v^{\epsilon,\psi^{\epsilon}}_{\eta^{\epsilon},n}(s,\cdot))-v^{\epsilon,\psi^{\epsilon}}_{\eta^{\epsilon}}(s,\cdot))\|_{\rho})\geq\frac{C}{2}\right)\\
&\leq\lim_{C\rightarrow\infty}\sup_{\epsilon\in[0,1]}\mathbb{P}\left(K\sup_{s\in[0,T]}\|v^{\epsilon,\psi^{\epsilon}}_{\eta^{\epsilon},n}(s,\cdot))\|_{\rho}\geq\frac{C}{4}\right)\\
&\leq\lim_{C\rightarrow\infty}\sup_{\epsilon\in[0,1]}\mathbb{P}\left(\|v^{\epsilon,\psi^{\epsilon}}_{\eta^{\epsilon},n}(s,\cdot))-v^{\epsilon,\psi^{\epsilon}}_{\eta^{\epsilon}}(s,\cdot))\|_{\rho})\geq\frac{C}{4}\right)=0.
\end{aligned}
\end{equation}
Hence, $J_{3}^{\epsilon}$ is tight. Similar to $J^{\epsilon}_{3}$, By assumption $\mathbf{(H4)}$, Minkowski's inequality and Young's inequality with exponent $\frac{\rho}{\rho-1}$ and $\frac{\rho}{2}$, we have
\begin{equation}
\begin{aligned}
\|J^{\epsilon}_{2}\|_{\rho}&\leq C\int_{0}^{t}\|\partial_{y}G_{t-s}(\cdot)\|_{\frac{\rho}{\rho-1}}\|g(s,\cdot,v^{\epsilon,\psi^{\epsilon}}_{\eta^{\epsilon}}(s,\cdot))\|_{\frac{\rho}{2}}ds\\
&\leq CT^{\frac{\rho-1}{2\rho}}\sup_{s\in[0,T]}\|g(s,\cdot,v^{\epsilon,\psi^{\epsilon}}_{\eta^{\epsilon}}(s,\cdot))\|_{\frac{\rho}{2}}\\
&\leq CT^{\frac{\rho-1}{2\rho}}(K+K\sup_{t\in[0,T]}\|v^{\epsilon,\psi^{\epsilon}}_{\eta^{\epsilon}}(s,\cdot)\|_{\rho}+K\sup_{t\in[0,T]}\|v^{\epsilon,\psi^{\epsilon}}_{\eta^{\epsilon}}(s,\cdot)\|_{\rho}^{2}).
\end{aligned}
\end{equation}
Let $\zeta^{\epsilon}=K+K\sup_{t\in[0,T]}\|v^{\epsilon,\psi^{\epsilon}}_{\eta^{\epsilon}}(s,\cdot)\|_{\rho}+K\sup_{t\in[0,T]}\|v^{\epsilon,\psi^{\epsilon}}_{\eta^{\epsilon}}(s,\cdot)\|_{\rho}^{2}$. For $J^{\epsilon}_{3}$, we  obtain
\begin{equation}
\begin{aligned}
\lim_{C\rightarrow\infty}\sup_{\epsilon\in[0,1]}\mathbb{P}(\zeta^{\epsilon})=\lim_{C\rightarrow\infty}\sup_{\epsilon\in[0,1]}\mathbb{P}\left(K+K\sup_{t\in[0,T]}\|v^{\epsilon,\psi^{\epsilon}}_{\eta^{\epsilon}}(s,\cdot)\|_{\rho}+K\sup_{t\in[0,T]}\|v^{\epsilon,\psi^{\epsilon}}_{\eta^{\epsilon}}(s,\cdot)\|_{\rho}^{2}\geq C\right)=0.
\end{aligned}
\end{equation}
Therefore, $J^{\epsilon}_{2}$ is also tight. For the tightness of the $J_{4}^{\epsilon}$, in view of the Cauchy-Schwartz inequality, Minkowski's inequality and Young's inequality, we have
\begin{equation}
\begin{aligned}
\|J_{4}^{\epsilon}\|_{\rho}&\leq \left(\int_{0}^{t}\int_{0}^{1}\left(\int_{0}^{y}\psi^\epsilon(s,y^{\prime})dy^\prime\right)^{2} dyds\right)^{\frac{1}{2}}\left\|\left(\int_{0}^{t}\int_{0}^{1}G^{2}_{t-s}(\cdot,y)\sigma^{2}(s,y,v^{\epsilon,\psi^{\epsilon}}_{\eta^{\epsilon}}(s,y))dyds\right)^{\frac{1}{2}}\right\|_{\rho}\\
&\leq C \left(\int_{0}^{t}\int_{0}^{1}|\psi^\epsilon(y,s)|^{2}dyds\right)^{\frac{1}{2}}\left(\int_{0}^{t}\left\|G^{2}_{t-s}(\cdot)\right\|_{1}\left\|\sigma(s,\cdot,v^{\epsilon,\psi^{\epsilon}}_{\eta^{\epsilon}})\right\|^{2}_{\rho}ds\right)^{\frac{1}{2}}\\
&\leq CT^{\frac{1}{2}}M^{\frac{1}{2}}(K+K\sup_{t\in[0,T]}\|v_{\eta^{\epsilon}}^{\epsilon,\psi^{\epsilon}}(t,\cdot)\|_{\rho}).
\end{aligned}
\end{equation}
Similar to the terms $J_{2}^{\epsilon}$ and $J_{3}^{\epsilon}$,  $J_{4}^{\epsilon}$ is also tight. Finally, we illustrate  the tightness of the stochastic convolution $J_{5}^{\epsilon}$, similar to Lemma \ref{tightness of stochatic term}, we can prove that $J_{5}^{\epsilon}$ is tight in $C([0,T]\times [0,1])$. Hence, it is tightness in $C([0,T];L^{\rho}([0,1]))$. Then $v^{\epsilon,\psi^{\epsilon}}_{\eta^{\epsilon}}$ is also tightness in $C([0,T];L^{\rho}([0,1]))$. According to Prokhorov's theorem, we can extract a subsequence of the controlled process $\{v^{\epsilon,\psi^{\epsilon}}_{\eta^{\epsilon}}\}$ converge to $v_{\eta}^{0,\psi}$ in distribution in $C([0,T];L^{\rho}([0,1]))$, let us write it down as $\{v^{\epsilon,\psi^{\epsilon}}_{\eta^{\epsilon}}\}$.  Hence we need to check  that $v_{\eta}^{0,\psi}$ satisfies \eqref{ske}, that is to say
\begin{equation}\nonumber
\begin{aligned}
&J_{1}^{\epsilon}\rightarrow\int_{0}^{1}G_{t}(x,y)\eta(y)dy\\
&J_{2}^{\epsilon}\rightarrow-\int_{0}^{t}\int_{0}^{1}\partial_{y}G_{t-s}(x,y)g(s,y,v^{0,\psi}_{\eta}(s,y))dyds\\
&J_{3}^{\epsilon}\rightarrow\int_{0}^{t}\int_{0}^{1}G_{t-s}(x,y)f(s,y,v^{0,\psi}_{\eta}(s,y))dyds\\
&J_{4}^{\epsilon}\rightarrow\int_{0}^{t}\int_{0}^{1}G_{t-s}(x,y)\sigma(s,y,v^{0,\psi}_{\eta}(s,y))\int_{0}^{y}\psi(s,y^\prime)dy^\prime dyds\\
&J_{5}^{\epsilon}\rightarrow 0
\end{aligned}
\end{equation}
in distribution in $C([0,T];L^{\rho}([0,1]))$. The convergence of $J^{\epsilon}_{1}$ is trivial. The Skorokhod representation theorem\cite[Theorem 3.2.2]{MR1876437} ensure that we can assume  that almost sure convergence on a common probability space. Indeed,  since $v^{\epsilon,\psi^\epsilon}_{\eta^\epsilon}$ is tight, then random elements $(v^{\epsilon,\psi^\epsilon}_{\eta^\epsilon},\frac{\partial W}{\partial x}, \psi^\epsilon)$ is also tight. Applying the Skorokhod representation theorem,  we can find a common probability space $(\Omega^{\prime},\mathcal{F}^{\prime},\mathbb{P}^{\prime})$  and  a sequences of random elements $\left(\mu_{\eta^{\epsilon},\frac{\tilde{W}^\epsilon}{\partial x}}^{\epsilon,\tilde{\psi}^{\epsilon}}, \frac{\tilde{W}^\epsilon}{\partial x}, \tilde{\psi}^{\epsilon}\right)$such that
$$\left(\mu_{\eta^{\epsilon},\frac{\tilde{W}^\epsilon}{\partial x}}^{\epsilon,\tilde{\psi}^{\epsilon}}, \frac{\tilde{W}^\epsilon}{\partial x}, \tilde{\psi}^{\epsilon}\right)=\left(v^{\epsilon,\psi^\epsilon}_{\eta^\epsilon},\frac{\partial W}{\partial x},\psi^\epsilon\right)$$
in distribution and $\mathbb{P}^\prime-$almost surely, $\left(\mu_{\eta^{\epsilon},\frac{\tilde{W}^\epsilon}{\partial x}}^{\epsilon,\tilde{\psi}^{\epsilon}}, \frac{\tilde{W}^\epsilon}{\partial x}, \tilde{\psi}^{\epsilon}\right)$ converges to a random element $\left(\mu^{0,\tilde{\psi}}_{\eta},\frac{\partial \tilde{W}}{\partial x},\tilde{\psi}\right)$ in the space
$$C([0,T];L^\rho([0,1]))\times C([0,T]\times[0,1];\mathbb{R})\times \mathcal{S}^{N}.$$
 In addition, $\mu_{\eta^{\epsilon},\frac{\tilde{W}^\epsilon}{\partial x}}^{\epsilon,\tilde{\psi}^{\epsilon}}$  satisfies the following equations
\begin{equation}\label{controlled equation1}
\begin{aligned}
\mu_{\eta^{\epsilon},\frac{\tilde{W}^\epsilon}{\partial x}}^{\epsilon,\tilde{\psi}^{\epsilon}}(t,x)=&\int_{0}^{1}G_{t}(x,y)\eta^{\epsilon}(y)dy-\int_{0}^{t}\int_{0}^{1}\partial_{y}G_{t-s}(x,y)g(s,y,\mu_{\eta^{\epsilon},\frac{\tilde{W}^\epsilon}{\partial x}}^{\epsilon,\tilde{\psi}^{\epsilon}}(s,y))dyds\\
&+\int_{0}^{t}\int_{0}^{1}G_{t-s}(x,y)f(s,y,\mu_{\eta^{\epsilon},\frac{\tilde{W}^\epsilon}{\partial x}}^{\epsilon,\tilde{\psi}^{\epsilon}}(s,y))dyds\\
&+\int_{0}^{t}\int_{0}^{1}G_{t-s}(x,y)\sigma(s,y,\mu_{\eta^{\epsilon},\frac{\tilde{W}^\epsilon}{\partial x}}^{\epsilon,\tilde{\psi}^{\epsilon}}(s,y))\int_{0}^{y}\tilde{\psi}^\epsilon(s,y^\prime)dy^\prime dyds\\
&+\sqrt{\epsilon}\int_{0}^{t}\int_{0}^{1}G_{t-s}(x,y)\sigma(s,y,\mu_{\eta^{\epsilon},\frac{\tilde{W}^\epsilon}{\partial x}}^{\epsilon,\tilde{\psi}^{\epsilon}}(s,y))\tilde{W}^\epsilon(dy,ds),
\end{aligned}
\end{equation}
and we can apply  $\mu_{\eta^{\epsilon},\frac{\tilde{W}^\epsilon}{\partial x}}^{\epsilon,\tilde{\psi}^{\epsilon}} \rightarrow \mu^{0,\tilde{\psi}}_{\eta}$ and $\tilde{\psi}^\epsilon\rightarrow\tilde{\psi}-$a.s.  to prove  that $ \mu^{0,\tilde{\psi}}_{\eta}$ satisfies the following equation
\begin{equation}\label{skeleton equation1}
\begin{aligned}
\mu^{0,\tilde{\psi}}_{\eta}(t,x)=&\int_{0}^{1}G_{t}(x,y)\eta(y)dy-\int_{0}^{t}\int_{0}^{1}\partial_{y}G_{t-s}(x,y)g(s,y,\mu^{0,\tilde{\psi}}_{\eta}(s,y))dyds\\
&+\int_{0}^{t}\int_{0}^{1}G_{t-s}(x,y)f(s,y,\mu^{0,\tilde{\psi}}_{\eta}(s,y))dyds\\
&+\int_{0}^{t}\int_{0}^{1}G_{t-s}(x,y)\sigma(s,y,\mu^{0,\tilde{\psi}}_{\eta}(s,y))\int_{0}^{y}\tilde{\psi}(s,y^\prime)dy^\prime dyds.
\end{aligned}
\end{equation}
Due to $\tilde{\psi}^\epsilon\rightarrow \tilde{\psi}$ and $\tilde{\psi}^\epsilon$ the same distribution as $\psi^\epsilon$, then $\tilde{\psi}$ has the same distribution as $\psi$. It follows from  the uniqueness of the solution of equation \eqref{skeleton equation1}  that  $\mu^{0,\tilde{\psi}}_{\eta}$  has the same distribution as $v^{0,\psi}_{\eta}=\mathcal{G}^0\left(\eta,\int_{0}^{t}\int_{0}^{x}\psi(s,x)dxds\right)$, then $v_{\eta^{\epsilon}}^{\epsilon,\psi^{\epsilon}}$ converge to  $v^{0,\psi}_{\eta}=\mathcal{G}^0\left(\eta,\int_{0}^{t}\int_0^x\psi(s,y)dyds\right)$ in distribution. For simplicity,  we further assume that $v_{\eta^{\epsilon}}^{\epsilon,\psi^{\epsilon}}\rightarrow v^{0,\psi}_{\eta}$ and $\psi^\epsilon\rightarrow\psi, \mathbb{P}^{\prime}-a.s.$, we need to check $v^{0,\psi}_{\eta}=\mathcal{G}^0\left(\eta,\int_{0}^{t}\int_0^x\psi(s,y)dyds\right)$ as $\mu^{0,\tilde{\psi}}_{\eta}=\mathcal{G}^0\left(\eta,\int_{0}^{t}\int_0^x\tilde{\psi}(s,y)dyds\right)$.  By Young's inequality with exponent $\frac{\rho}{2}$ and $\frac{\rho}{\rho-1}$, Cauchy-Schwartz's inequality and assumption $\textbf{(H3)}$, we have
\begin{equation}
\begin{aligned}
&\left\|J_{2}^{\epsilon}+\int_{0}^{t}\int_{0}^{1}\partial_{y}G_{t-s}(x,y)g(s,y,v^{0,\psi}_{\eta}(s,y))dyds\right\|_{\rho}\nonumber\\
&~~=\left\|\int_{0}^{t}\int_{0}^{1}\partial_{y}G_{t-s}(\cdot,y)\left(g(s,y,v_{\eta^{\epsilon}}^{\epsilon,\psi^{\epsilon}})-g(s,\cdot,v^{0,\psi}_{\eta})\right)dyds\right\|_{\rho}
\\&~~\leq C \int_{0}^{t}(t-s)^{-1+\frac{\rho-1}{2\rho}}\left\|(1+|v_{\eta^{\epsilon}}^{\epsilon,\psi^{\epsilon}}(s,\cdot)|+|v^{0,\psi}_{\eta}(s,\cdot)|)(|v_{\eta^{\epsilon}}^{\epsilon,\psi^{\epsilon}}(s,\cdot)-v^{0,\psi}_{\eta}(s,\cdot)|)\right\|_{\frac{\rho}{2}}ds\\
&~~\leq C\int_{0}^{t}(t-s)^{-1+\frac{\rho-1}{2\rho}} (1+\|v_{\eta^{\epsilon}}^{\epsilon,\psi^{\epsilon}}(s,\cdot)\|_{\rho}+\|v^{0,\psi}_{\eta}(s,\cdot)\|_{\rho})\|v_{\eta^{\epsilon}}^{\epsilon,\psi^{\epsilon}}(s,\cdot)-v^{0,\psi}_{\eta}(s,\cdot)\|ds\\
&~~\leq C\int_{0}^{t}(t-s)^{-1+\frac{\rho-1}{2\rho}}ds(1+\sup_{t\in[0,T]}\|v_{\eta^{\epsilon}}^{\epsilon,\psi^{\epsilon}}(t,\cdot)\|_{\rho}+\sup_{t\in[0,T]}\|v^{0,\psi}_{\eta}(t,\cdot)\|_{\rho}) \sup_{t\in[0,T]} \|v_{\eta^{\epsilon}}^{\epsilon,\psi^{\epsilon}}(t,\cdot)-v^{0,\psi}_{\eta}(t,\cdot)\|_{\rho}.
\end{aligned}
\end{equation}
Since $v_{\eta^{\epsilon}}^{\epsilon,\psi^{\epsilon}}\rightarrow v^{0,\psi}_{\eta}$ almost surely as $\epsilon\rightarrow 0$, then the above inequality converge to zero almost surely. Hence, we have
\begin{equation}
J_{2}^{\epsilon}\rightarrow-\int_{0}^{t}\int_{0}^{1}\partial_{y}G_{t-s}(x,y)g(s,y,v^{0,\psi}_{\eta}(s,y))dyds.
\end{equation}
For $J^{\epsilon}_{3}$, we can obtain the convergence of $J^{\epsilon}_{3}$ in similar way. By the assumption $\mathbf{(H3)}$, Young's inequality and Cauchy-Schwartz's inequality
\begin{equation}
\begin{aligned}
&\left\|J_{3}^{\epsilon}-\int_{0}^{t}\int_{0}^{1}G_{t-s}(x,y)f(s,y,v^{0,\psi}_{s,y})dyds\right\|_{\rho} =\left\|\int_{0}^{t}\int_0^{1}G_{t-s}(\cdot,y)(f(s,y,v_{\eta^{\epsilon}}^{\epsilon,\psi^{\epsilon}}-f(s,y,v^{0,\psi}_{\eta}))dyds\right\|\\
&\leq C\int_{0}^{t}(t-s)^{-\frac{1}{2}+\frac{\rho-1}{2\rho}}ds(1+\sup_{t\in[0,T]}\|v_{\eta^{\epsilon}}^{\epsilon,\psi^{\epsilon}}(t,\cdot)\|_{\rho}+\sup_{t\in[0,T]}\|v^{0,\psi}_{\eta}(t,\cdot)\|_{\rho}) \sup_{t\in[0,T]} \|v_{\eta^{\epsilon}}^{\epsilon,\psi^{\epsilon}}(t,\cdot)-v^{0,\psi}_{\eta}(t,\cdot)\|_{\rho}.
\end{aligned}
\end{equation}
Hence, we have
\begin{equation}
J_{3}^{\epsilon}\rightarrow\int_{0}^{t}\int_{0}^{1}G_{t-s}(x,y)f(s,y,v^{0,\psi}_{\eta}(s,y))dyds.
\end{equation}
As for  the convergence of $J^{\epsilon}_{4}$, by triangle inequality, we obtain
\begin{equation}
\begin{aligned}
&\left\|J^{\epsilon}_{4}-\int_{0}^{t}\int_{0}^{1}G_{t-s}(x,y)\sigma(s,y,v^{0,\psi}_{\eta}(s,y))\int_{0}^{y}\psi(s,y^\prime)dy^\prime dyds\right\|\\
&\leq \left\|\int_{0}^{t}\int_{0}^{1}|G_{t-s}(\cdot,y)||\sigma(s,y,v_{\eta^{\epsilon}}^{\epsilon,\psi^{\epsilon}}(s,y))-\sigma(s,y,v^{0,\psi}_{\eta})(s,y)|\left|\int_{0}^{y}\psi^\epsilon(s,y^\prime)dy^\prime\right|dyds\right\|_{\rho}\\
&+\left\|\int_{0}^{t}\int_{0}^{1}|G_{t-s}(\cdot,y)||\sigma(s,y,v^{0,\psi}_{\eta})|\left|\int_{0}^{y}\psi^\epsilon(s,y^\prime)-\psi(s,y^{\prime})dy^\prime\right|dyds\right\|_{\rho}\\
&:=J^{\epsilon,1}_{4}+J^{\epsilon,2}_{4}.
\end{aligned}
\end{equation}
For $J^{\epsilon,1}_{4}$, in view of Cauchy-Schwartz's inequality,  Minkowski's inequality, Young's inequality and $\mathbf{(H2)}$, we have
\begin{equation}
\begin{aligned}
\|J_{4}^{\epsilon,1}\|_{\rho} &\leq CT^{\frac{1}{2}}M^{\frac{1}{2}}\left\|\left(\int_{0}^{t}\int_{0}^{1}G^{2}_{t-s}(x,y)\left(\sigma(s,y,v_{\eta^{\epsilon}}^{\epsilon,\psi^{\epsilon}}(s,y))-\sigma(s,y,v^{0,\psi}_{\eta}(s,y))\right)^{2}dyds\right)^{\frac{1}{2}}\right\|_{\rho}\\
&\leq CT^{\frac{1}{2}} M^{\frac{1}{2}}\left(\int_{0}^{t}\|G^{2}_{t-s}(\cdot)\|_{1}\|(\sigma(s,\cdot,v_{\eta^{\epsilon}}^{\epsilon,\psi^{\epsilon}}(s,\cdot))-\sigma(s,\cdot,v_{\eta}^{0,\psi}(s,\cdot)))^{2}\|ds\right)^{\frac{1}{2}}\\
&\leq C T^{\frac{1}{2}}M^{\frac{1}{2}}\left(\int_{0}^{t}(t-s)^{-\frac{1}{2}}ds\right)^{\frac{1}{2}}\sup_{t\in [0,T]}\|v_{\eta^{\epsilon}}^{\epsilon,\psi^{\epsilon}}(s,\cdot)-v_{\eta}^{0,\psi}(s,\cdot)\|_{\rho}.
\end{aligned}
\end{equation}
Thus, $J^{\epsilon,1}_{4}\rightarrow 0$ almost surely. For  $J^{\epsilon,2}_{4}$, by $\psi^{\epsilon}\rightarrow \psi$ weakly in $\mathcal{S}^{N}$, H\"{o}lder's inequality,  Minkowski's inequality, Young's inequality and $\mathbf{(H2)}$, we obtain
\begin{equation}
\begin{aligned}
\|J^{\epsilon,2}_{4}\|_{\rho}&\leq \left(\int_{0}^{t}\int_{0}^{1}(\psi^{\epsilon}(s,y)-\psi(s,y))^{2}dyds\right)^{\frac{1}{2}}\left\|\left(\int_{0}^{t}\int_{0}^{1}G^{2}_{t-s}(x,y)\sigma^{2}(s,y,v^{0,\psi}_{\eta}(s,y))dyds\right)^{\frac{1}{2}}\right\|_{\rho}\\
&\leq\left(\int_{0}^{t}\int_0^1(\psi^\epsilon(s,y)-\psi(s,y))^2dyds\right)^{\frac{1}{2}}\left(\int_{0}^{t}(t-s)^{-\frac{1}{2}}ds\right)^{\frac{1}{2}}\sup_{t\in[0,T]}(1+\|v_{\eta}^{0,\psi}(t,\cdot)\|_{\rho}).
\end{aligned}
\end{equation}
Thus we get $J_{4}^{\epsilon,2}\rightarrow 0$ almost surely. Then we have
\begin{equation}
J_{4}^{\epsilon}\rightarrow\int_{0}^{t}\int_{0}^{1}G_{t-s}(x,y)\sigma(s,y,v^{0,\psi}_{\eta}(s,y))\int_{0}^{y}\psi(s,y^\prime)dy^\prime dyds.
\end{equation}
Finally, $J^{\epsilon}_{5}$ has a version $C([0,T]\times[0,1];\mathbb{R})$ as Lemma \ref{tightness of stochatic term}, then $J^{\epsilon}_{5}$ converges to zero almost surely. We can obtain that the controlled process
converge to the skeleton equation for  a given subsequence. Note that for any sequence $\epsilon_{n}$, we have the same result. So according to $\epsilon_n$ are arbitrary , then $v^{\epsilon,\psi^\epsilon}_{\eta^\epsilon}\rightarrow v^{0,\psi}_{\eta}=\mathcal{G}^{0}\left(\eta,\int_{0}^{t}\int_{0}^{x}\psi(s,y)dyds\right)$. \\

Finally, we need to check the condition $(a)$, it suffices to combine with the well-posedness of the skeleton equation and the result of convergence of $v^{0,\psi^{\epsilon}}_{\eta^{\epsilon}}$ to check it. So, the proof of Theorem 6.1 is completed.
\end{proof}
\textbf{Conflict of interest:} The authors declare that they have no conflict of interest.

\bibliographystyle{abbrv}

\end{document}